\documentclass[12pt]{article}
\usepackage{amsmath,amssymb,amsfonts,amsthm}
\usepackage[dvips]{graphics}
\newcommand{\BB}{\mathbb{B}}
\newcommand{\C}{\mathbb{C}}

\newcommand{\Z}{\mathbb{Z}}
\newcommand{\R}{\mathbb{R}}
\newcommand{\G}{{\mathcal{G}}}
\newcommand{\N}{{\mathcal{N}}}
\newcommand{\E}{\mathbf{E}}
\renewcommand{\H}{\mathcal H}
\newcommand{\Id}{\text{Id}}
\newcommand{\GFF}{\text{GFF}}
\newcommand{\A}{\mathbb{A}}
\newcommand{\eps}{\epsilon}
\renewcommand{\b}{\mathrm{b}}
\newcommand{\w}{\mathrm{w}}
\newtheorem{theorem}{Theorem}
\newtheorem{lemma}[theorem]{Lemma}
\newtheorem{cor}[theorem]{Corollary}
\newtheorem{exer}{Exercise}
\newcommand{\im}{\text{Im}}
\newcommand{\re}{\text{Re}}

\begin{document}
\title{Lectures on dimers}
\author{Richard Kenyon}
\date{} \maketitle

\tableofcontents
\section{Overview}

The planar dimer model is, from one point of view,
a statistical mechanical model of random 
$2$-dimensional interfaces in $\R^3$. In a concrete sense it
is a natural generalization of the simple random walk on $Z$. 
While the simple random walk and its scaling limit, Brownian
motion, permeate all of probability theory and many other
parts of mathematics, higher dimensional models like the dimer model
are much less used or understood. 
Only recently have tools been developed for gaining a mathematical
understanding of two-dimensional random fields. The dimer
model is at the moment the most successful of these
two-dimensional theories. What is remarkable is that
the objects underlying the simple random walk: the Laplacian,
Green's function, and Gaussian measure, are also
the fundamental tools used in the dimer model. 
On the other hand the study of the dimer model actually uses
tools from many other areas of mathematics:
we'll see a little bit of algebraic
geometry, PDEs, analysis, and ergodic theory, at least.

Our goal in these notes is to study the planar dimer model and the 
associated random interface model.
There has been a great deal of recent research on the dimer
model but there are still many interesting open questions
and
new research avenues. 
In these lecture notes we will provide an introduction to the dimer model, leading up to 
results on limit shapes and fluctuations. 
There are a number of exercises and a few open 
questions at the end of each chapter. Our objective is not to give
complete proofs, but we will at least attempt to indicate 
the ideas behind many of 
the proofs. Complete proofs can all be found in the papers
in the bibliography. The main references we used for these notes are
\cite{KOS, KO, Bout, deTil}. The original papers of Kasteleyn \cite{Kast}
and Temperley/Fisher \cite{TF} are quite readable.

Much previous work on the dimer model will not be discussed here;
we'll focus here on dimers on planar, bipartite, periodic
graphs. Of course one can study dimer models in
which all or some of these assumptions have been weakened.
While these more general models have been discussed and
studied in the literature, there are fewer concrete results
in this greater generality. It is our hope that these notes
will inspire work on these more general problems.

\subsection{Dimer definitions}

A {\bf dimer covering}, or {\bf perfect matching}, of a graph is a subset of edges which covers every vertex exactly once, that is, every vertex is the endpoint of exactly one edge.  See Figure \ref{dom}.

In these lectures we will deal only with {\bf bipartite planar} graphs.
A graph is bipartite when the vertices can be colored black and white in such
a way that each edge connects vertices of different colors. Alternatively, this is equivalent to each cycle having even length.
Kasteleyn showed how to enumerate the dimer covers of any planar graph,
but the random surface interpretation we will discuss is valid only for 
bipartite graphs.
There are many open problems involving dimer coverings of non-bipartite planar graphs,
which at present we do not have tools to attack. However at
present we have some nice tools to deal with periodic bipartite planar graphs.

Our prototypical examples are the dimer models on $\Z^2$
and the honeycomb graph. These are equivalent to,
respectively, the domino tiling model 
(tilings with $2\times 1$ rectangles) and the ``lozenge tiling'' model
(tilings with $60^\circ$ rhombi) see Figures \ref{dom} and \ref{loz}.
\begin{figure}[htbp]
\center{\scalebox{1.0}{\includegraphics{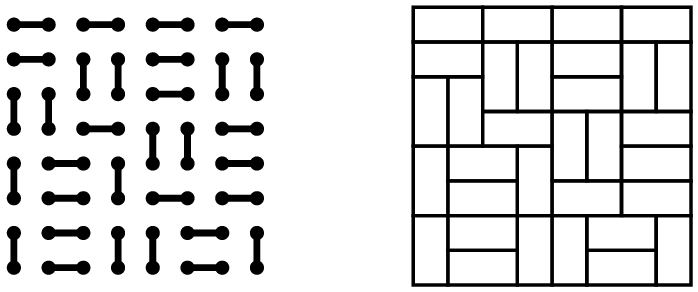}}}
\caption{\label{dom}}
\end{figure}
\begin{figure}[htbp]
\center{\scalebox{0.4}{\includegraphics{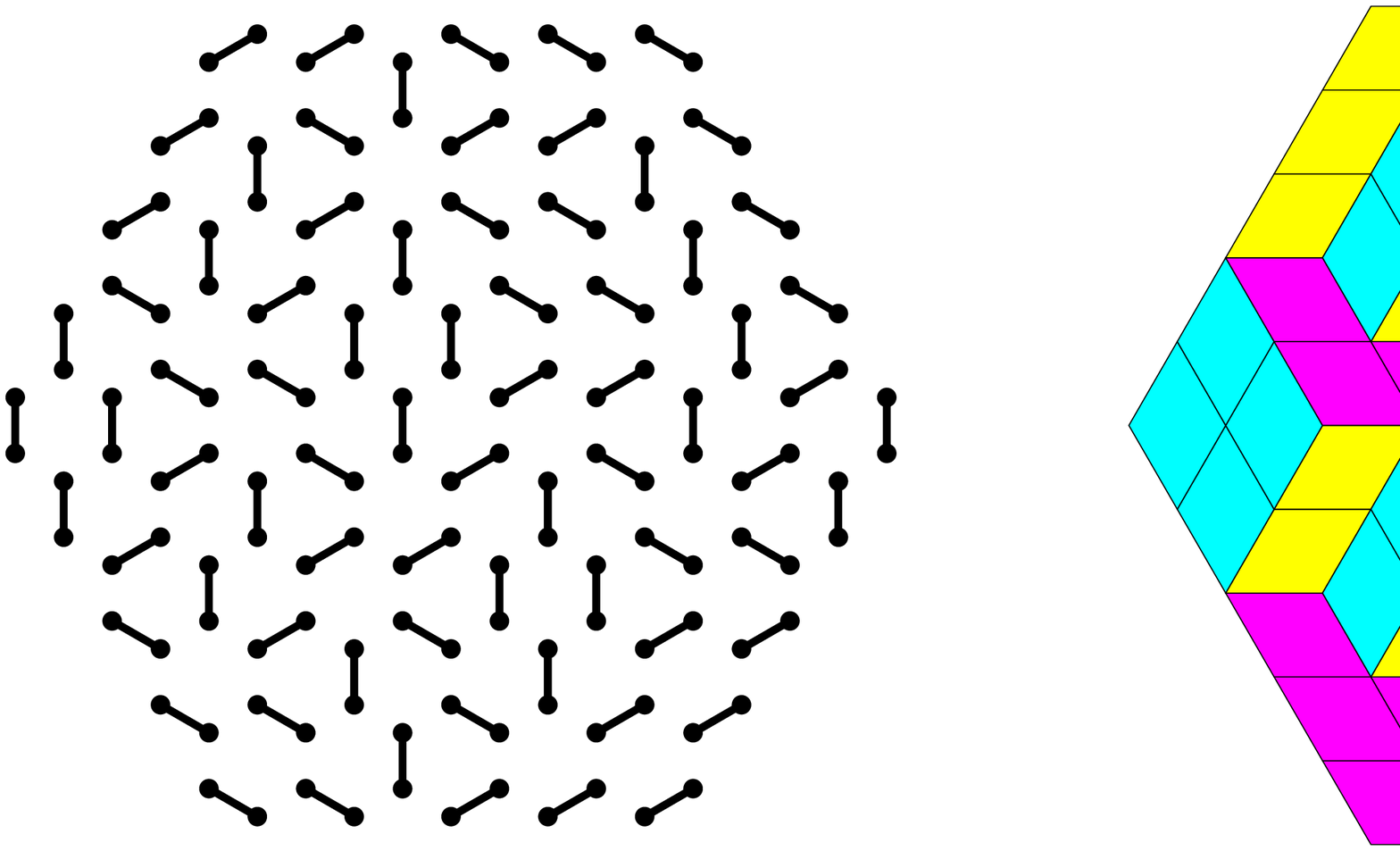}}}
\caption{\label{loz}}
\end{figure}

Dimers on the honeycomb have been studied in chemistry \cite{KHDD,HZK}
where they are called Kekul\'e structures. The honeycomb is after all
the structure of graphite, each carbon atom sharing one double bond
with a neighbor.

\subsection{Uniform random tilings}

Look at a larger domino picture and the lozenge picture, Figures \ref{dominosquare} and
\ref{lozengehexagon}.
\begin{figure}[htbp]
\scalebox{0.7}{\includegraphics{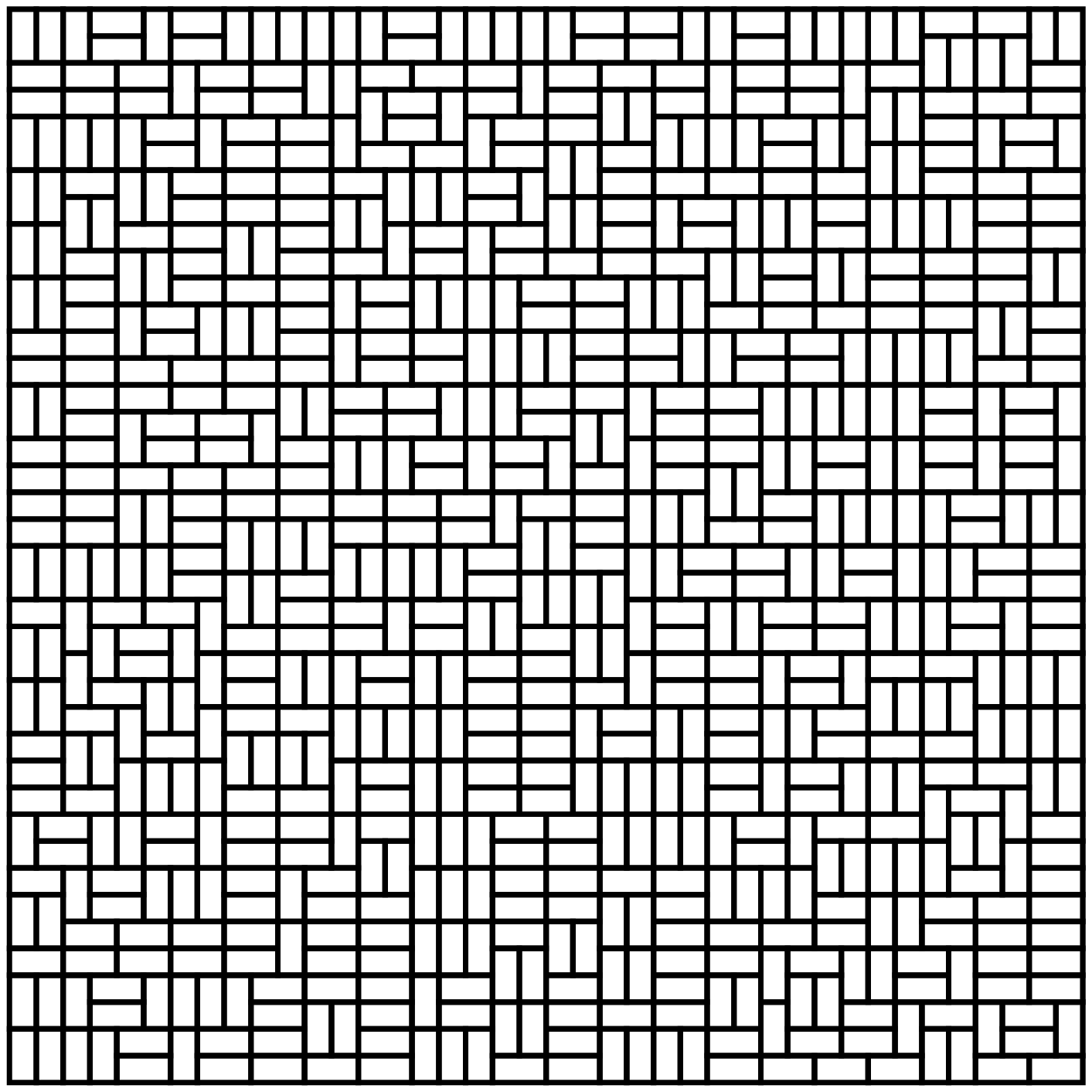}}
\caption{\label{dominosquare}}
\end{figure}
\begin{figure}[htbp]
\scalebox{0.9}{\includegraphics{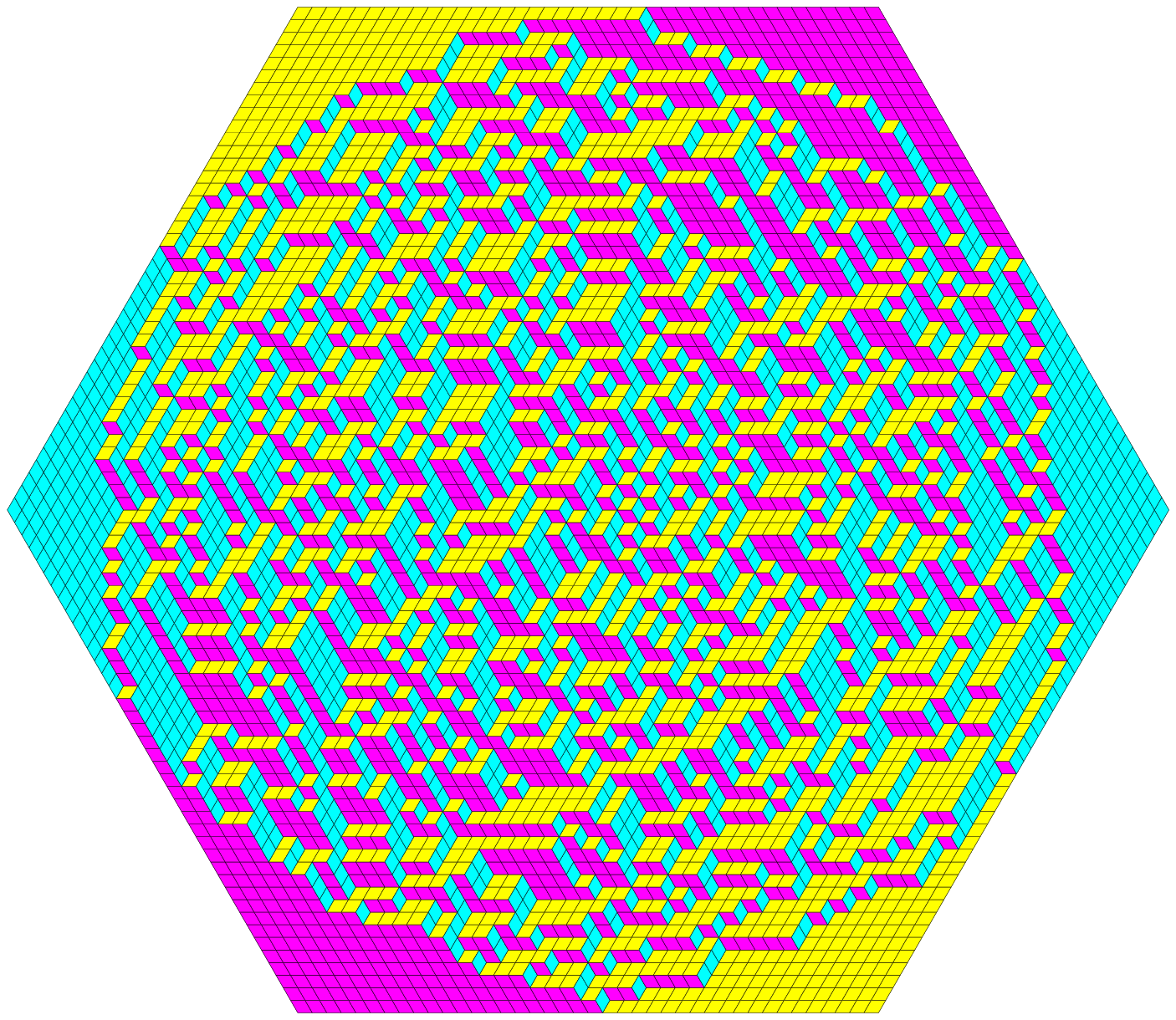}}
\caption{\label{lozengehexagon}}
\end{figure}
These are both uniform random tilings of the corresponding regions,
that is, they are chosen from the distribution in which all tilings are
equally weighted.
In the first case there are about $e^{455}$ 
possible domino tilings and in the second,
about $e^{1255}$ lozenge tilings \footnote{How do you pick a 
random sample from such a large space?}
These two pictures clearly display some very different behavior.
The first picture appears homogeneous 
(and we'll prove that it is, in a well-defined sense), 
while in the second,
the densities of the tiles of a given orientation vary throughout
the region. A goal of these lectures is to understand this phenomenon,
and indeed compute the limiting densities as well as other statistics
of these and other tilings, in a setting of reasonably general boundary conditions.

In figures \ref{lozengeonside1} and \ref{lozengeonside2}, we see a uniform random lozenge tiling
of a triangular shape, and the same tiling rotated so that we see
(with a little imagination)
the fluctuations. These fluctuations are quite small, in fact of order
$\sqrt{\log n}$ for a similar triangle of side $n$. 
\begin{figure}[htbp]
\center{\scalebox{1.0}{\includegraphics{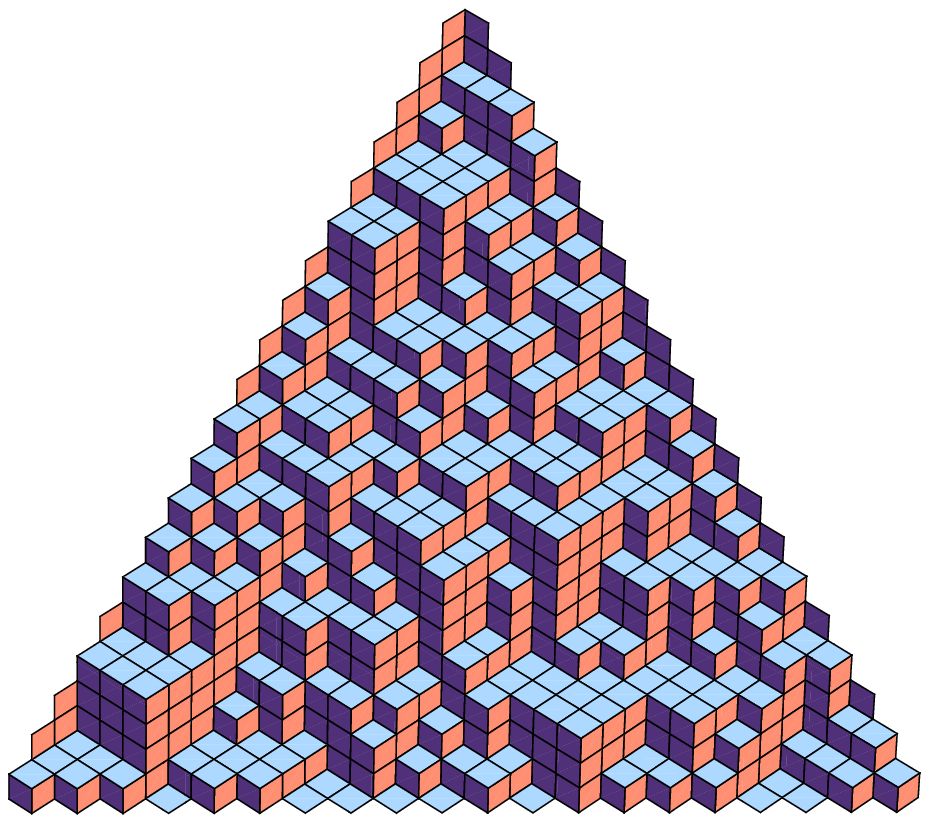}}}
\caption{\label{lozengeonside1}}
\end{figure}
\begin{figure}[htbp]
\scalebox{1.0}{\includegraphics{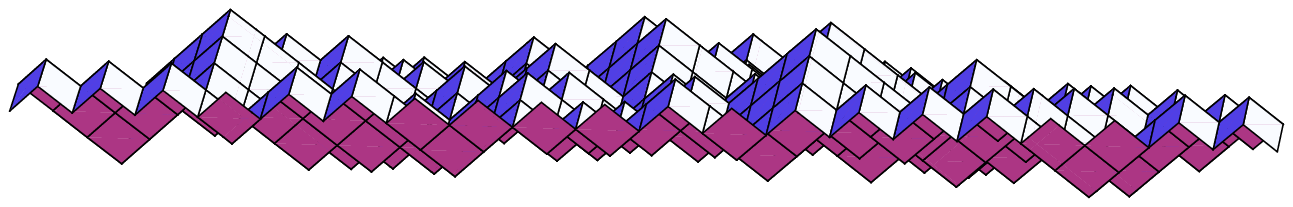}}
\caption{\label{lozengeonside2}}
\end{figure}
This picture should be compared with Figure \ref{SRW} which shows 
the one-dimensional analog of the lozenge tiling of a triangle.
It is just the graph of a simple random walk on $\Z$ of length $n=100$
conditioned to start and end
at the origin. In this case the fluctuations are of order $\sqrt{n}$. Indeed,
if we rescale the vertical coordinate by $\sqrt{n}$ and the horizontal 
coordinate by $n$, the resulting curve converges as $n\to\infty$
to a {\bf Brownian
bridge} (a Brownian motion started at the origin and conditioned to return
to the origin after time $1$). 
\begin{figure}[htbp]
\center{\scalebox{1.0}{\includegraphics{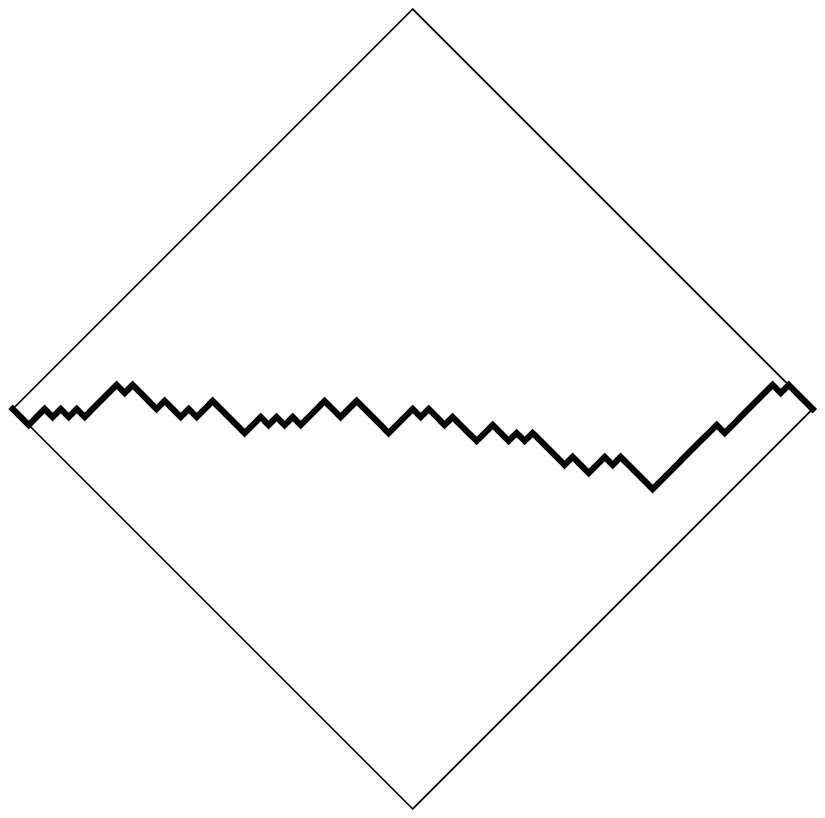}}}
\caption{\label{SRW}}
\end{figure}

The ``scaling limit" of the fluctuations of the lozenge tiling of a triangle is a more complicated object, called
the Gaussian free field. We can think of it as a Gaussian random function
but in fact it is only a random distribution (weak function). 
We'll talk more about it later.

\subsection{Limit shapes}

A lozenge tiling of a simply connected region is the projection
along the $(1,1,1)$-direction of
a piecewise linear surface in $\R^3$. This surface has pieces which 
are integer translates of the sides of the unit cube.  
Such a surface which projects injectively in the $(1,1,1)$-direction is called
a {\bf stepped surface} or {\bf skyscraper surface}.
Random lozenge tilings are therefore random stepped surfaces.
A random lozenge tiling of a fixed region as in Figure \ref{lozengehexagon}
is a random stepped surface spanning a fixed boundary curve in $\R^3$. 

Domino tilings can also be interpreted as random tilings. Here the third coordinate
is harder to visualize, but see Figure \ref{dominoheight} for a definition.
\begin{figure}[htbp]
\scalebox{1.0}{\includegraphics{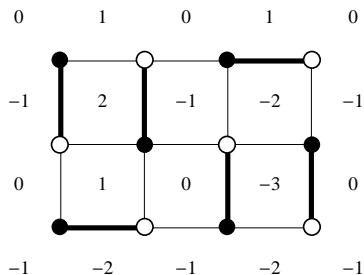}}
\caption{\label{dominoheight} The height is integer-valued on the
faces and changes by $\pm1$ across an unoccupied edge; the change is $+1$ when crossing so that a white vertex is on the left.}
\end{figure}
We'll see below that dimers on any bipartite graph can be interpreted as random surfaces. 

When thought of as random surfaces, these models have a 
{\bf limit shape phenomenon}. This says that, for a 
fixed boundary curve in $\R^3$,
or sequence of converging boundary curves in $\R^3$, if we take a random stepped surface on finer and finer lattices, then with probability tending to $1$
the random surface will lie close to a fixed non-random surface,
the so-called {\bf limit shape}. So this limit shape surface is not just the average surface
but the only surface you will see if you take an extremely fine mesh...the measure
is concentrating as the mesh size tends to zero
on the delta-measure at this surface.

Again the analogous property in one dimension is that the random curve of 
Figure \ref{SRW} tends to a straight line when the mesh size tends to zero.
(In order to see the fluctuations one needs to scale the two coordinates at
different rates. Here we are talking about scaling both coordinates 
by $n$: the fluctuations are zero on this scale.)

As Okounkov puts it, the limit shape surface is in some sense the ``most random''
surface, in the sense that among all surfaces with the given boundary values,
the limit shape is the one which has (overwhelmingly) 
the most discrete stepped surfaces nearby, in fact so many as to dwarf all 
other surfaces combined.

\subsection{Facets}

One thing to notice about lozenge tiling in figure 
\ref{lozengehexagon} is the presence
of regions near the vertices of the hexagon where the lozenges
are aligned. This phenomenon persists in the limit of small
mesh size and in fact, in the limit shape surface there is a {\bf facet}
near each corner, where the limit shape is planar. 
This is a phenomenon which does not occur in one dimension
(but see however \cite{NO}). The limit shapes for dimers
generally contain facets
and smooth (in fact analytic) curved regions separating these facets. 
In the facet the probability of a misaligned lozenge is zero in the limit of small mesh, and in fact
one can show that the probability is exponentially small in the 
reciprocal of the mesh size.
In particular the fluctuations away from the facet are exponentially small.

\subsection{Measures} 

What do we see if we zoom in to a point in figure \ref{lozengehexagon}? That is, consider a sequence of such figures with
the same fixed boundary but decreasing mesh size. Pick a point in the 
hexagon and consider the configuration restricted to a small window
around that point, window which gets smaller as the mesh size goes to zero. One can imagine for example a window of side $\sqrt{\eps}$
when the mesh size is $\eps$. This gives a sequence of
random tilings of (relative to the mesh size) larger and larger domains, and in the limit (assuming that a limit of these ``local measures'' exists)
we will get a random tiling of the plane. 

We will see different types of
behaviors, depending on which point we zoom in on. If we zoom in to a point
in the facet, we will see in the limit a boring figure in which all tiles are aligned.
This is an example of a measure on lozenge tilings of the plane
which consists of a delta measure at a single tiling. This measure 
is an (uninteresting) example of an ergodic Gibbs measure (see definition below). 
If we zoom into a point in the non-frozen region, 
one can again ask what
limiting measure on tilings of the plane is obtained. One of the important open problems
is to understand this limiting measure, in particular to 
prove that the limit exists. Conjecturally it exists and only depends
on the slope of the average surface at that point and not on any other
property of the boundary conditions.
For each possible slope $(s,t)$ we will define below a measure $\mu_{s,t}$,
and the {\bf local statistics conjecture} states that, for any 
fixed boundary,  $\mu_{s,t}$ is the 
measure which occurs in the limit at any point 
where the limit shape has slope $(s,t)$. 
For certain boundary conditions this has been proved \cite{Kenyon.fluct}. 

These measures $\mu_{s,t}$ are discussed in the next section.

\subsubsection{Ergodic Gibbs Measures}

What are the natural probability measures on lozenge tilings of the whole plane? 
One might require for example that the measure be translation-invariant.
In this case one can associate a {\bf slope} or gradient to the measure, which is the expected
change in height when you move in a given direction.

Another natural condition to impose on a probability measure on tilings of the plane is that it be a {\bf Gibbs measure}, that is, (in this case)
a probability measure which is the limit of the uniform
measure on tilings of finite regions, as the regions increase in size to fill out the whole plane. There is a generalization of this definition
to dimers with edge weights.

A remarkable theorem of Sheffield states that for each slope $(s,t)$
for which there
is an invariant measure, there is a unique ergodic Gibbs measure
(ergodic means not a convex combination of other invariant measures).
We denote this measure $\mu_{s,t}$.

\subsubsection{Phases}

Ergodic Gibbs measures come in three types, or {\bf phases}, 
depending on 
the fluctuations of a typical (for that measure) surface. 
Suppose we fix the height at a face near the origin to be zero.
A measure is said to be a {\bf frozen phase} if the height 
fluctuations are
finite almost surely, that is, the fluctuation of the surface away
from its mean value is almost surely bounded, no matter how far away from the origin
you are.
A measure is said to be in a {\bf liquid phase} if the fluctuations 
have variance increasing with increasing distance, that is, 
the variance of the height at a point tends to infinity almost surely
for points farther and farther from the origin.
Finally a measure is said to be a {\bf gaseous phase}
if the height fluctuations  
are unbounded
but the variance of the height at a point is bounded
independently of the distance from the origin.
 
We'll see that for uniform lozenge or domino tilings we can have both liquid and frozen
phases but not gaseous phases. An example of a graph for which we have a 
all three phases is the square octagon dimer model, see Figure \ref{squareoctagon}.
\begin{figure}[htbp]
\center{\scalebox{0.6}{\includegraphics{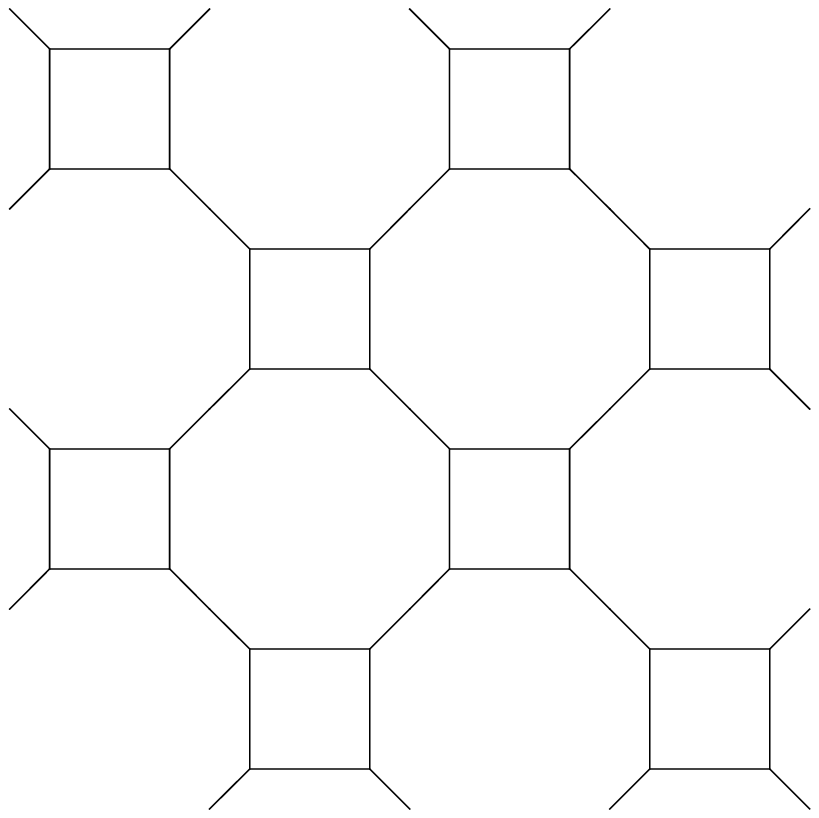}}}
\caption{\label{squareoctagon}}
\end{figure}
In general the classification of phases depends on algebraic properties
of the underlying graph and edge weights, as we'll see.

\subsection{Other random surface models}

There are a number of other statistical mechanical models
with similar behavior to the planar dimer model. 
The most well-known of these is
the six-vertex model, or ``square ice'' model. This model consists of
orientations of the edges of the square grid (or other planar graph of degree $4$)
with the restriction that at each vertex there are two ``ingoing''
and two ``outgoing'' arrows. 
This model also has an interpretation as a random interface model, with a limit shape phenomenon, facet formation, and so on.
The dimer model is fundamentally easier to deal with
than this model, and other models of this type, essentially because enumeration of dimer configurations
can be accomplished using determinants, while the $6$-vertex model
requires more complicated and mysterious algebraic structures. 
Indeed, the quantities which have been computed for
these other models are been quite limited, 
and no one has succeeded in writing down a 
formula for a limit shape in any of these other models.

\section{The height function}

We show here how to define the height function for dimers on any bipartite graph.
This allows us to give a ``random surface" interpretation for dimers
on any planar bipartite graph, using the height function as the third coordinate.

\subsection{Graph homology}

Let $\G$ be a connected 
planar graph or connected graph embedded on a surface.
Let $\Lambda^0=\Lambda^0(\G)$ be the space of functions on the vertices of $\G$.

A {\bf flow}, or {\bf $1$-form}, is a function $\omega$
on oriented
edges of $\G$ which is antisymmetric under changing orientation:
if $e$ is an edge connecting vertices $v_1$ and $v_2$
and oriented from $v_1$ to $v_2$,
then $\omega(e)=-\omega(-e)$ where by $-e$ we denote the same edge oriented
from $v_2$ to $v_1$. Let $\Lambda^1$ be the space of
$1$-forms on $\G$.

We define a linear operator $d\colon\Lambda^0\to\Lambda^1$ by
$dg(v_1v_2)=g(v_2)-g(v_1)$.
The transpose of the operator $d$ (for the natural bases 
indexed by vertices and edges) is $d^*$, the {\bf divergence}.
The divergence $d^*\omega$ of a flow $\omega$ is a function on vertices 
defined by $$d^*\omega(v)=\sum_{e} \omega(e)$$
where the sum is over edges starting at $e$. A positive divergence at $v$
means $v$ is a source: water is flowing into the graph at $v$. A negative
divergence is a sink: water is flowing out of the graph at $v$.

We define
a {\bf $2$-form} to be a function on oriented faces which is antisymmetric
under changing orientation. Let $\Lambda^2$ be the space of 
$2$-forms.
We define $d:\Lambda^1\to\Lambda^2$, the {\bf curl},
by: for a $1$-form $\omega$
and face $F$,  $d\omega(F)=\sum_e\omega(e)$, where the sum is over edges bounding $F$ oriented consistently (cclw)
with the orientation of $F$. 
The $1$-forms $\omega$ for which $d\omega=0$ are {\bf cocycles}.
The $1$-forms in the image of $d\Lambda^0$ are {\bf coboundaries},
or {\bf gradient flows}.

Note that $d\circ d$ as a map from $\Lambda^0$ to $\Lambda^2$
is identically zero. Moreover 
for planar embeddings, the space of cocycles modulo the space of coboundaries is
trivial, which is to say that
every cocycle is a coboundary.
Equivalently, the sequence
$$0\rightarrow \Lambda^0\stackrel{d}{\rightarrow}\Lambda^1\stackrel{d}{\rightarrow}\Lambda^2\rightarrow 0$$
is exact. 

For a nonplanar graph embedded on a surface in such a way that
every face is
a topological disk, the space of cocycles modulo the space of coboundaries 
is the $1$-homology of the surface. In particular for a graph embedded
in this way on a torus, the $1$-homology is $\R^2$.
The $1$-forms
which are nontrivial in homology are those which have non-zero
net flux across the horizontal and/or vertical loops around the torus.
The point in $\R^2$ whose coordinates give these two net fluxes is the homology class of the $1$-form.

\subsection{Heights}

Given a dimer cover $M$ of a bipartite graph, there is a naturally associated flow $\omega_M$:
it has value $1$ on each edge occupied by a dimer, when the edge is
oriented from its white vertex to its black vertex. Other edges have flow $0$.
This flow has divergence $d^*\omega_M=\pm1$ 
at white, respectively black vertices.

In particular given two dimer covers of the same graph, the difference
of their
flows is a divergence-free flow. 

Divergence-free flows on planar graphs are dual to gradient flows,
that is, 
Given a divergence-free flow and a fixed face $f_0$, one can define 
a function $h$ on all faces of $\G$ as follows: $h(f_0)=0$,
and for any other face $f$, $h(f)$ is the net flow crossing
(from left to right) a path in the dual graph from $f_0$ to $f$. 
The fact that the flow is divergence-free means that $h(f)$ does
not depend on the path from $f_0$ to $f$. 

To define the height function for a dimer cover, we proceed as follows.
Fix a flow $\omega_0$ with divergence $1$ at white vertices and divergence $-1$
at black vertices ($\omega_0$ might come from a fixed dimer cover but this
is not necessary). Now for any dimer cover $M$,
let $\omega_M$ be its corresponding flow. Then the difference 
$\omega_M-\omega_0$
is a divergence-free flow. Let $h=h(M)$ be the corresponding
function on faces of $\G$ (starting from some fixed face $f_0$). Then $h(M)$ is the {\bf height function}
of $M$. It is a function on faces of $\G$.
Note that if $M_1,M_2$ are two dimer coverings then
$h(M_1)-h(M_2)$ does not depend on the choice of $\omega_0$. In particular
$\omega_0$ is just a ``base point" for $h$ and the more natural quantity
is really the difference $h(M_1)-h(M_2)$. 

For lozenge tilings a natural choice of base point flow is 
the flow $\omega_0$ with value $1/3$ on each edge (oriented from white to black). In this case it is not hard to see that
$h(M)$ is just, up to an additive constant, the height (distance from the plane $x+y+z=0$)
when we think of a lozenge tiling
as a stepped surface. Another base flow, $\omega_a$, 
which is one
we make the most use of below, is the flow which is $+1$
on all horizontal edges and zero on other edges.
For this base flow the three axis planes (that is, the dimer covers
using all edges of one orientation) have slopes 
$(0,0),(1,0)$ and $(0,1)$.

Note that for a bounded subgraph in the honeycomb,
the flow $\omega_0$ or $\omega_a$
will typically not have divergence $\pm1$ at white/black vertices
on the boundary. This just means that the boundary height function is
nonzero: thought of as a lozenge tiling, the boundary curve is 
not flat.

Similarly, for domino tilings a natural choice for $\omega_0$ is $1/4$ on each edge, oriented from white to black.
This is the definition of the height function for dominos as defined by Thurston
\cite{Thu} and also ($1/4$ of) that in Figure \ref{dominoheight}. 
Thurston used it to give a linear-time algorithm for tiling simply-connected
planar regions with dominos (that is deciding whether or not
a tiling exists and building one if there is one). 
 
For a nonplanar graph embedded on a surface, the flows can be defined as above, but not the height functions in general. This is because there
might be some {\bf period}, or height change, along topologically
nontrivial closed dual loops. This period is however a homological
invariant in the sense that two homologous loops will have the same
period. For example for a graph embedded on a torus,
we can define two periods $h_x,h_y$ for loops in homology
classes $(1,0)$ and $(0,1)$ respectively. Then any loop in 
homology class $(x,y)\in\Z^2$ will have height change $xh_x+yh_y$.

\begin{exer}
{}From an $8\times 8$ checkerboard, two squares are removed
from opposite corners. In how many ways can the resulting figure be
tiled with dominos?
\end{exer}

\begin{exer}
Using the height function defined in Figure \ref{dominoheight},
find the lowest and highest tilings of a $2m\times 2n$ rectangle.
\end{exer}

\begin{exer} Take two hexagons with edge lengths 
$11,13,11,13,11,13$
in cyclic order, and glue them together along one of their sides of
length $11$. Can the resulting figure be tiled with lozenges?
What if we glue along a side of length $13$?
\end{exer}

\begin{exer}
{}From a $2n\times 2n$ square, a subset of the squares on the lower
edge are removed. For which subsets does there exist a domino
tiling of the resulting region?
\end{exer}

\begin{exer} For a finite graph $\G$,
let $\Omega$ be the set of 
flows $\omega$ satifying $\omega(e)\in[0,1]$
when $e$ is oriented from white to black,
and having divergence $\pm 1$ and white/black vertices,
respectively.  Show that $\Omega$ is a convex polytope. 
Show that the
dimer coverings of $\G$ are the vertices of this polytope. In particular
if $\Omega\neq\emptyset$ then there is a dimer cover. 
\end{exer}

\section{Kasteleyn theory}

We show how to compute the number of dimer coverings of any bipartite planar graph using the KTF (Kasteleyn-Temperley-Fisher)  technique.
While this technique extends to nonbipartite planar graphs,
we will have no use for this generality here.

\subsection{The Boltzmann measure}

To a finite bipartite planar graph $\G=(B\cup W,E)$
with a positive real function $w:E\to\R_{>0}$ 
on edges, we define a probablity measure $\mu=\mu(\G,w)$ on dimer covers
by, for a dimer covering $M$, 
$$\mu(M)=\frac1Z\prod_{e\in M}w(e)$$
where the sum is over all dimers of $M$ and
where $Z$ is a normalization constant called the {\bf partition function}, 
defined to be
$$Z=\sum_M\prod_{e\in M}w(e).$$

\subsection{Gauge equivalence}

If we change the weight function $w$ by multiplying edge weights of all edges incident to a single vertex $v$ by a constant
$\lambda$, then $Z$ is multiplied by $\lambda$, and
the measure $\mu$ does not change, since each dimer cover
uses exactly one of these edges. 
So we define two weight functions $w,w'$ to be equivalent, $w\sim w'$,
if one can be obtained from the other by a sequence of such
multiplications. 
It is not hard to show that $w\sim w'$ if and only if 
the {\bf alternating products} along faces are equal:
given a face with edges $e_1,e_2,\dots,e_{2k}$ in cyclic
order, the quotients $$\frac{w(e_1)w(e_3)\dots w(e_{2k-1})}{w(e_2)w(e_4)
\dots w(e_{2k})}$$ and $$\frac{w'(e_1)w'(e_3)\dots w'(e_{2k-1})}{w'(e_2)w'(e_4)
\dots w'(e_{2k})}$$ (which we call alternating products) must be equal.
For nonplanar graphs, gauge equivalence is the same as having equal
alternating products along all cycles (and it suffices to consider 
cycles in a homology basis). 

The proof is easy, and in fact is just a little homology theory:
$\log w$ is a $1$-form (if you orient edges from white to black)
and $w$ is equivalent to $w'$
if and only if $\log w-\log w'=df$ for some function $f\in\Lambda^0$.
However $\log w-\log w'=df$ for a planar graph is equivalent to  
$d\log w=d\log w'$ which can be interpreted as saying that the alternating
products of $w$ and $w'$ on cycles are equal.

\subsection{Kasteleyn weighting}\label{Kweighting}

A {\bf Kasteleyn weighting} of a planar bipartite graph
is a choice of sign for each undirected edge with the property that  
each face with $0\bmod 4$ edges has an odd number of $-$ signs and
each face with $2\bmod 4$ edges has an even number of $-$ signs.

In certain circumstances it will be convenient to use
complex numbers of modulus $1$ rather than signs $\pm1$.
In this case the condition is that the alternating product of 
edge weights (as defined above) around a face is negative real or positive real
depending on whether the face has $0$ or $2 \bmod 4$ edges.

This condition appears mysterious at first but we'll see why it is important below\footnote{The condition might appear more natural
if we note that the alternating product is required to be $e^{\pi iN/2}$
where $N$ is the number of triangles in a triangulation of the face}.
We can see as in the previous section that any two
Kasteleyn weightings are gauge equivalent: they can be obtained
one from the other by a sequence of operations consisting
of multiplying all edges at a vertex by a constant.

The existence of a Kasteleyn weighting is also easily established
for example using spanning trees. We leave this fact to the reader,
as well as the proof of the following (easily proved by induction)
\begin{lemma}\label{signchange}
Given a cycle of length $2k$ enclosing $\ell$ points,
the alternating product of signs around this cycle is
$(-1)^{1+k+\ell}$.
\end{lemma}

Note finally that for the (edge-weighted) honeycomb graph, 
all faces have $2\bmod 4$ edges and so no signs are necessary
in the Kasteleyn weighting.

\subsection{Kasteleyn matrix}

A {\bf Kasteleyn matrix} is a weighted, signed adjacency matrix
of the graph $\G$. Given a Kasteleyn weighting of $\G$, define
a $|B|\times |W|$ matrix $K$ by
$K(\b,\w)=0$ if there is no edge from $\w$ to $\b$,
otherwise $K(\b,\w)$ is the Kasteleyn weighting times the edge weight 
$w(\b\w)$. 

For the graph in Figure \ref{3X1} with Kasteleyn weighting 
indicated, the Kasteleyn matrix is 
$$\left(\begin{matrix}a&1&0\\1&-b&1\\0&1&c\end{matrix}\right).$$
\begin{figure}[htbp]
\scalebox{1.0}{\includegraphics{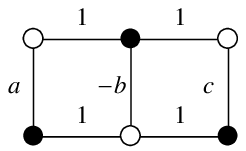}}
\caption{\label{3X1}}
\end{figure}
Note that gauge transformation corresponds to pre- or post-multiplication
of $K$ by a diagonal matrix. 

\begin{theorem}[\cite{Kast, TF}]\label{Z}
$ Z=|\det K|.$
\end{theorem}

In the example, the determinant is $-a-c-abc$.

\begin{proof}
If $K$ is not square the determinant is zero and there are no dimer coverings
(each dimer covers one white and one black vertex).
If $K$ is a square $n\times n$ matrix,
We expand 
\begin{equation}\label{detexpansion}
\det K=\sum_{\sigma\in S_n}\text{sgn}(\sigma)
K(\b_1,\w_{\sigma(1)})\dots K(\b_n,\w_{\sigma(1)}).
\end{equation}
Each term is zero unless it pairs each black vertex with a unique
neighboring white vertex. So there is one term for each dimer
covering, and the modulus of this term is the product of its edge weights.
We need only check that the signs of the nonzero terms are all equal.

Let us compare the signs of two different nonzero terms.
Given two dimer coverings, we can draw them simultaneously on $\G$.
We get a set of doubled edges and loops. To convert one
dimer covering into the other, we can take a loop and
move every second dimer (that is, dimer from the first covering)
cyclically around by one edge so that 
they match the dimers from the second covering. When we do this operation 
for a single loop of length $2k$, we are changing the permutation $\sigma$ by a 
$k$-cycle. Note that by Lemma \ref{signchange} above the sign change of the edge weights
in the corresponding term in 
(\ref{detexpansion}) is $\pm1$ depending on whether $2k$ is
$2\bmod 4$ or $0\bmod 4$ (since $\ell$ is even there), 
exactly the same sign change as occurs in
$\text{sgn}(\sigma)$. These two sign changes cancel, showing that
these two coverings (and hence any two coverings) have the same sign.
\end{proof}

A simpler proof of this theorem for honeycomb graphs---which avoids 
using Lemma \ref{signchange}---goes as follows: if two dimer coverings differ
only on a single face, that is, an operation of the type 
in Figure \ref{hexflip} converts one cover into the other,
then these coverings have the same sign in the expansion of the determinant, because the hexagon 
flip changes $\sigma$ by a $3$-cycle
which is an even permutation. 
\begin{figure}[htbp]
\center{\scalebox{1.0}{\includegraphics{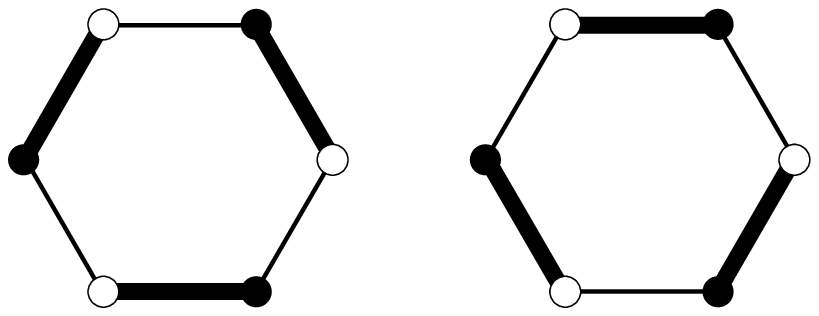}}}
\caption{\label{hexflip}}
\end{figure}
Thus it suffices to notice
that any two dimer coverings can be obtained
from one another by a sequence of 
hexagon flips. This can be seen using the lozenge tiling picture since
applying a hexagon flip is equivalent to adding or subtracting a cube
from the stepped surface. Any two surfaces
with the same connected boundary can be obtained from one another by adding and/or subtracting cubes.

While there is a version of Theorem \ref{Z} (using Pfaffians) for non-bipartite
planar graphs, there is no corresponding sign trick for nonplanar graphs 
in general (the exact condition is that a graph has a Kasteleyn
weighting if and only if it does not have $K_{3,3}$ as minor \cite{LP}).

\subsection{Local statistics}

There is an important corollary to Theorem \ref{Z}:

\begin{cor}[\cite{Kenyon.localstats}]\label{localstatcor}
Given a set of edges $X=\{\w_1\b_1,\dots,\w_k\b_k\}$,
the probability that all edges in $X$ occur in a dimer cover is
$$\left(\prod_{i=1}^k K(\b_i,\w_i)\right)\det (K^{-1}(\w_i,\b_j))_{1\leq i,j\leq k}.$$
\end{cor}

The proof uses the 
Jacobi Lemma that says that a minor of a matrix $A$ is $\det A$
times the complementary minor of $A^{-1}$. 

The advantage of this result is that the probability
of a set of $k$ edges being present is only a $k\times k$ determinant, independently
of the size of the graph. One needs only be able to compute $K^{-1}$.
In fact the corollary is valid even for infinite graphs,
once $K^{-1}$ has been appropriately defined. 

Corollary
\ref{localstatcor} shows that the edges form a {\bf determinantal
(point) process}. Such a process is defined by the fact that
the probability of a set of $k$ ``points'' $p_1,\dots,p_k$
is a determinant of a $k\times k$
matrix $M$ with entries $M(p_i,p_j)$. Here the points in the process
are edges of $\G$ and $M(p_i,p_j)=K(\b_i,\w_i)K^{-1}(\w_i,\b_j)$
where $\w_i$ is the white vertex of edge $p_i$ and $\b_j$ is the black
vertex of edge $p_j$. See \cite{Sosh} for an introduction to
determinantal processes.

\begin{exer}
Without using Kasteleyn theory, compute the number of 
domino tilings of a $2\times n$ and $3\times n$ rectangle.
\end{exer}

\begin{exer}\label{bppexer}
A classical combinatorial result of Macmahon says that the number
of lozenge tilings of an $A\times B\times C$ hexagon (that is,
hexagon with edges $A,B,C,A,B,C$ in cyclic order) is
$$\prod_{i=1}^A\prod_{j=1}^B\prod_{k=1}^C \frac{i+j+k-1}{i+j+k-2}.$$
What is the probability that, in a uniform random tiling,
there are two lozenges adjacent to a chosen corner of an
$n\times n\times n$ hexagon\footnote{In fact a more general formula holds:
If we weight configurations with $q^{\mathrm{volume}}$, then 
$$Z=\prod_{i=1}^A\prod_{j=1}^B\prod_{k=1}^C \frac{1-q^{i+j+k-1}}{1-q^{i+j+k-2}}.$$}?
\end{exer}

\begin{exer}
Show that a planar bipartite graph has a Kasteleyn weighting
and prove Lemma \ref{signchange}.
\end{exer}

\begin{exer}
Verify Corollary \ref{localstatcor} for the example in Figure \ref{3X1}.
\end{exer}

\section{Partition function}

\subsection{Rectangle}
Here is the simplest example. 
Assume $mn$ is even and
let $\G_{m,n}$ be the $m\times n$ square grid.
Its vertices are $V=\{1,2,\dots,m\}\times\{1,2,\dots, n\}$
and edges connect nearest neighbors.
Let $Z_{m,n}$ be the partition function for dimers with
edge weights $1$. This is just the number of dimer coverings of
$\G_{m,n}$.

A Kasteleyn weighting is obtained by putting weight $1$
on horizontal edges and $i=\sqrt{-1}$ on vertical edges.
Since each face has four edges, the condition in
(\ref{Kweighting}) is satisfied.\footnote{A weighting gauge equivalent to this one,
and using only weights $\pm1$, is
to weight alternate columns of vertical edges by $-1$ and
all other edge $+1$. This was the weighting originally used by Kasteleyn \cite{Kast}; our current weighting (introduced by Percus \cite{Percus})
is slightly easier for our purposes.}

The corresponding Kasteleyn matrix 
$K$ is an $mn/2\times mn/2$ matrix 
(recall that $K$ is a $|W|\times|B|$ matrix). The eigenvalues of the matrix
$\tilde K=\left(\begin{matrix}0&K\\K^t&0\end{matrix}\right)$
are in fact simpler to compute. Let $z=e^{i\pi j/(m+1)}$
and $w=e^{i\pi k/(n+1)}$.
Then the function 
$$f_{j,k}(x,y)=(z^x-z^{-x})(w^y-w^{-y})=-4\sin(\frac{\pi jx}{m+1})
\sin(\frac{\pi ky}{n+1})$$
is an eigenvector of $\tilde K$ with eigenvalue
$z+\frac1z+i(w+\frac1w)$. 
To see this, check that
$$\lambda f(x,y)=f(x+1,y)+f(x-1,y)+if(x,y+1)+if(x,y-1)$$
when $(x,y)$ is not on the boundary of $\G$, and also true when
$f$ is on the boundary assuming we extend $f$ to be zero 
just outside the boundary, i.e. when $x=0$ or $y=0$ or $x=m+1$
or $y=n+1$.

As $j,k$ vary in $(1,m)\times(1,n)$ the eigenfunctions $f_{j,k}$
are independent (a well-known fact from Fourier series).
Therefore we have a complete diagonalization of the matrix $\tilde K$,
leading to 
\begin{equation}\label{Zproduct}
Z_{m,n} =\left(\prod_{j=1}^m\prod_{k=1}^n 2\cos\frac{\pi j}{m+1}+2i
\cos\frac{\pi k}{n+1}\right)^{1/2}.
\end{equation}
Here the square root comes from the fact that 
$\det \tilde K=(\det K)^2.$

Note that if $m,n$ are both odd then this expression is zero
because of the term $j=(m+1)/2$ and $k=(n+1)/2$ in
(\ref{Zproduct}).

For example $Z_{8,8}=12988816$.
For large $m,n$ we have
$$\lim_{m,n\to\infty}\frac{1}{mn}\log Z_{m,n}=\frac{1}{2\pi^2}\int_0^\pi\int_0^\pi
\log(2\cos\theta+2i\cos\phi)d\theta\, d\phi$$
which can be shown to be equal to $G/\pi$,
where $G$ is Catalan's constant $G=1-\frac1{3^2}+\frac1{5^2}-\dots$.

\begin{exer}
Show this.
\end{exer}

\subsection{Torus}\label{Torus}

A graph on a torus does not in general have a Kasteleyn 
weighting. However we can still make a ``local'' Kasteleyn
matrix whose determinant can be used to count dimer covers.

Rather than show this in general, let us work out a detailed example.
Let $H_n$ be the honeycomb lattice on a torus,
as in Figure \ref{honeycombtorus}, which shows $H_3$. 
\begin{figure}[htbp]
\center{\scalebox{1.0}{\includegraphics{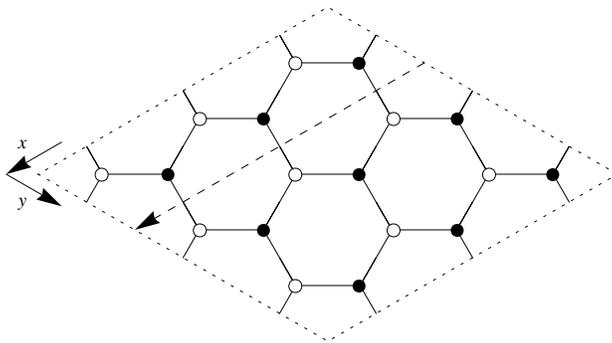}}}
\caption{\label{honeycombtorus}Fundamental domain for the
honeycomb graph on a torus.}
\end{figure}It has $n^2$ black vertices and
$n^2$ white vertices, and $3n^2$ edges.
Weight edges $a,b,c$
according to direction: $a$ the horizontal, $b$ the NW-SE edges
and $c$ the NE-SW edges. Let $\hat x$ and $\hat y$ be the directions
indicated, and $h_x$ and $h_y$ be the height changes along the paths winding around the torus in these directions, defined
using as base flow $\omega_a$ (the flow coming from the ``all-$a$'' dimer cover).
In particular, $h_x, h_y$ are the number of $b$-, respectively $c$-type edges crossed on a path 
winding around the torus in the $\hat x$, respectively $\hat y$
direction.
 
We have the following lemma.

\begin{lemma} The height change $(h_x,h_y)$ of a dimer covering
of $H_n$ is determined by the numbers $N_a,N_b,N_c$
of dimers of each
type $a,b,c$, as follows: $h_x=\frac{N_b}n$  and 
$h_y=\frac{N_c}n$. 
\end{lemma}

\begin{proof}
We can compute the height change $h_x$ along any path winding around the torus in the $x$-direction. Take for example the dashed path 
in Figure \ref{honeycombtorus}. The height change $h_x$ is the number of
dimers crossing this path.
We get the same value $h_x$ for any of the $n$ translates 
of this path to other $y$ coordinates. Summing over all translates 
gives $n h_x=N_b$ and dividing by $n$ gives the result.
The same argument holds in the $y$-direction.
\end{proof}

Let $K_n$ be the weighted adjacency
matrix of $H_n$. That is $K(\b,\w)=0$ if there is no edge
from $\b$ to $\w$ and otherwise $K(\b,\w)=a,b$ or $c$
according to direction. 

{}From the proof of Theorem \ref{Z} we can see that
$\det K$ is a weighted, signed sum of dimer coverings.
Our next goal is to determine the signs.

\begin{lemma}
The sign of a dimer covering in $\det K$ depends only on its
height change $(h_x,h_y)$ modulo $2$. Three of the four parity
classes gives the same sign and the fourth has the opposite sign.
\end{lemma}

\begin{proof} Let $N_b,N_c$ be the number of $b$ and $c$
type edges in a cover. 
If we take the union of a covering with the covering consisting of all 
$a$-type edges, we can compute the sign of the cover
by the product of the sign changes when shifting along each loop. The number of loops is $q=\text{GCF}(h_x,h_y)$ and each of
these has homology class $(h_y/q,h_x/q)$
(note that the $b$ edges contribute to $h_x$ but to the $(0,1)$
homology class). The length of each
loop is $\frac{2n}{q}(h_x+h_y)$ and so each loop contributes
sign $(-1)^{1+\frac{n}{q}(h_x+h_y)}$, for a total of 
$(-1)^{q+n(h_x+h_y)}$.
Note that $q$ is even if and only if $h_x$ and $h_y$ are both even.
So if $n$ is even then the sign is $-1$ unless $(h_x,h_y)\equiv (0,0)\bmod 2$. If $n$ is odd the sign is $+1$ unless 
$(h_x,h_y)\equiv (1,1)\bmod 2$.
\end{proof}

In particular $\det K$, when expanded as a polynomial in $a,b$
and $c$, has coefficients which count coverings 
with particular height changes. 
For example for $n$ odd we can write
$$\det K=\sum (-1)^{h_xh_y}C_{h_x,h_y}a^{n(n-h_x-h_y)}b^{nh_x}c^{nh_y}$$
since $h_xh_y$ is odd exactly when $h_x,h_y$ are both odd.

Define $Z_{00}=Z_{00}(a,b,c)$ to be this expression. 
Define 
$$Z_{10}(a,b,c)=Z_{00}(a,be^{\pi i/n},c),$$
$$Z_{01}(a,b,c)=Z_{00}(a,b,ce^{\pi i/n}),$$
and 
$$Z_{11}(a,b,c)=Z_{00}(a,be^{\pi i/n},ce^{\pi i/n}).$$

Then one can verify that, when $n$ is odd,
\begin{equation}\label{ZZZZ}
Z=\frac12(Z_{00}+Z_{10}+Z_{01}-Z_{11})
\end{equation}
which is equivalent to Kasteleyn's expression \cite{Kast} for the partition function. The case $n$ even is similar and left to the reader.

\subsection{Partition function}

Dealing with a torus makes computing the determinant much easier,
since the graph now has many translational symmetries.
The matrix $K$ commutes with translation operators and
so can be simultaneously diagonalized with them.
(In other words, we can use a Fourier basis.)
Simultaneous eigenfunctions of the horizontal and vertical
translations (and $K$) are the exponential functions
$f_{z,w}(x,y)=z^{-x} w^{-y}$ where $z^n=1=w^n$. 
There are $n$ choices for $z$ and $n$ for $w$ leading to a complete
set of eigenfunctions. The corresponding eigenvalue for $K$
is $a+bz+cw$.

In particular 
$$\det K =Z_{00}=\prod_{z^n=1}\prod_{w^n=1} a+bz+cw.$$

This leads to 
$$Z_{10}=\prod_{z^n=-1}\prod_{w^n=1} a+bz+cw.$$
$$Z_{01}=\prod_{z^n=1}\prod_{w^n=-1} a+bz+cw.$$
$$Z_{11}=\prod_{z^n=-1}\prod_{w^n=-1} a+bz+cw.$$

\subsection{Height change distribution}

{}From the expression for $Z_{00}$ one can try to estimate,
for $a,b,c=1$, the size of the various coefficients.
Assuming that $n$ is a multiple of $3$,
the largest coefficient $C_{h_x,h_y}$ occurs when $(h_x,h_y)=(n/3,n/3)$,
that is, $(N_a,N_b,N_c)=(\frac{n^2}3,\frac{n^2}3,\frac{n^2}3)$.
Boutillier and de Tili\'ere \cite{BdT} showed that, 
letting $C=C_{n/3,n/3}$ and $(h_x,h_y)=(j+\frac{n}3,k+\frac{n}3),$ 
$$C_{h_x,h_y}=Ce^{-c(j^2+jk+k^2)}(1+O(1/n))$$ where $c$ is a constant.
So for large $n$ the height change of a uniform random tiling 
on a torus
has a distribution converging to a discrete Gaussian distribution 
centered at $(n/3,n/3)$.

\begin{exer}
What signs should we put in (\ref{ZZZZ}) in the case $n$ is even?
\end{exer}

\begin{exer}
Compute the partition function for dimers on a square grid 
on a cylinder of circumference $4n+2$ and height $m$. Weight 
horizontal edges $a$ and vertical edges $b$.
\end{exer}

\section{Gibbs measures}

\subsection{Definition}
Let $X=X(\G)$ be the set of dimer coverings of a graph $\G$, possibly infinite,
with edge weight function $w$.
Recall the definition of the Boltzmann probability measure on $X(\G)$ 
when $\G$ is finite: a dimer
covering has probability proportional to the product of its edge weights.
When $\G$ is infinite, this definition will of course not work. 
For an infinite graph $\G$, a probability measure on $X$ is a {\bf Gibbs measure} 
if it is 
a weak limit of Boltzmann measures on a sequence of finite
subgraphs of $\G$ filling out $\G$. 
By this we mean, for any finite subset of $\G$, the probability of
any particular configuration occurring on this finite set 
converges. That is, the probabilities of {\bf cylinder sets} converge.
Here a cylinder set is a subset of $X(\G)$ consisting of all
coverings which have a particular configuration on a fixed set of vertices.
For example the set of coverings containing a particular edge is a cylinder set.

For a sequence of Boltzmann measures on increasing graphs, 
the limiting measure may not exist, but subsequential limits will 
always exist.
The limit may not be unique, however; that is,
it may depend on the approximating sequence of finite graphs.
This will be the case for dimers and it is this non-uniqueness which
makes the dimer model interesting.

The important property of Gibbs measures is the following.
Let $A$ be a cylinder set defined by the presence of a finite set of edges.
Let $B$ be another cylinder set defined by a different set of edges but
using the same set of vertices. Then, for the approximating Boltzmann measures, the ratio of the measures of 
these two cylinder sets is equal to the ratio of the product of the edge weights
in $A$ and the product of the edge weights in $B$. This is 
true along the finite growing sequence of graphs and so in particular
the same is true for the limiting Gibbs measures. 
In fact this property {\em characterizes} Gibbs measures: given a finite
set of vertices, the measure on this set of vertices conditioned
on the exterior (that is, integrating over the configuration in the exterior) is just the Boltzmann measure on this finite set.

\subsection{Periodic graphs}

We will be interested in the case when $\G$ is a {\bf periodic 
bipartite planar graph}, that is,
a planar bipartite weighted graph on which translations in $\Z^2$ 
(or some other rank-$2$ lattice) act by
weight-preserving and color-preserving isomorphisms.
Here by {\bf color-preserving} isomorphisms,
we mean, isomorphism which maps white vertices to white and black vertices to black. 
Note for example that for the graph $\G=\Z^2$ with nearest neighbor edges,
the lattice generated by $(2,0)$ and $(1,1)$ acts by color-preserving
isomorphisms, but $\Z^2$ itself does not act by color-preserving
isomorphisms. So the fundamental domain 
contains two vertices, one white and one black. 

For simplicity we will assume our periodic graphs are embedded so that
the lattice of weight- and color-preserving isomorphisms is $\Z^2$, 
so that we can describe a translation using a pair of integers.

\subsection{Ergodic Gibbs measures}

For a periodic graph $\G$, a {\bf translation-invariant measure} 
on $X(\G)$ is simply
one for which the measure
of a subset of $X(\G)$ in invariant under the translation-isomorphism action. 

The {\bf slope} $(s,t)$ of a translation-invariant measure is
the expected height change in the $(1,0)$ and $(0,1)$ directions,
that is, $s$ is the expected height change between a face and its translate by $(1,0)$
and $t$ is the expected height change between a face and its translate by $(0,1).$
 
A {\bf ergodic} Gibbs measure, or {\bf EGM}, is one in which 
translation-invariant sets have measure $0$ or $1$.
Typical examples of translation-invariant sets are: the set of coverings
which contain a translate of a particular pattern.

\begin{theorem}[Sheffield \cite{Sheff}]\label{Sheffthm}
For the dimer model on a periodic planar bipartite periodically edge-weighted graph,
for each slope $(s,t)$ for which there exists a translation-invariant measure,
there exists a unique EGM $\mu_{s,t}$. Moreover every EGM is of this type for some $s,t$.
\end{theorem}

In particular we can classify EGMs by their slopes.

The existence is not hard to establish by taking limits of Boltzmann measures on
larger and larger tori
with restricted height changes $(h_x,h_y)$, see below. 
The uniqueness is much harder; we won't discuss this here.

\subsection{Constructing EGMs}

Going back to our torus $H_n$, note that the number of
$a,b,$ and $c$ type edges are multiples of $n$ and satisfy 
$$0\leq \frac{N_a}n,\frac{N_b}n,\frac{N_c}n\leq n$$
$$ \frac{N_a}n+\frac{N_b}n+\frac{N_c}n=n$$
and in fact it is not hard to see that every triple of integers in this triangle 
can occur. Recalling that the height changes were related to these by
$h_x=N_b/n, h_y=N_c/n$, and the average slope $(s,t)$ is defined
by $(h_x/n,h_y/n)$, we have that 
$(s,t)$ lies in the triangle 
$$\{(s,t)~:~0\leq s,t,s+t\leq 1\}.$$

We conclude from Theorem \ref{Sheffthm}
that there is a unique EGM for every $(s,t)$ in this triangle.
We denote this EGM $\mu_{s,t}$.

We can construct $\mu_{s,t}$ as follows.
On $X(H_n)$ restrict to configurations for which $(h_x,h_y)=
(\lfloor ns\rfloor, \lfloor nt\rfloor)$. Let $\mu^{(n)}_{s,t}$ be the
Boltzmann measure on $X(H_n)$ conditioned on this set.
Any limit of the $\mu^{(n)}_{s,t}$ as $n\to\infty$
is a translation-invariant Gibbs measure of slope $(s,t)$. 
Ergodicity follows from uniqueness: the set of translation-invariant
Gibbs measures of slope $(s,t)$ is convex and its extreme points
are the ergodic ones; since there is a unique ergodic one this
convex set must be reduced to a point.

In practice, it is hard to apply this construction, since conditioning
is a tricky business. The next section gives an alternate construction.

\subsection{Magnetic field}

The weights $a,b,c$ are playing dual roles.
On the one hand they are variables in the partition function
whose exponents determine the height change of a covering.
On the other hand by putting positive real values in for
$a,b,c$ we {\bf reweight} the different coverings.
This reweighting has the property that it depends only
on $h_x,h_y$, that is, two configurations with the same
$h_x,h_y$ are reweighted by the same quantity.
As a consequence putting in weights $a,b,c$ has the effect
of changing the average value of $(h_x,h_y)$ for a random
dimer cover of $T_n$. However a random dimer cover of
a {\em planar} region is unaffected by this reweighting since the
Boltzmann measures are unchanged.

There is in fact a law of large numbers for covers of $T_n$:
as the torus gets large, the slope of a random tiling
(with edges weights $a,b,c$) 
is concentrating on a fixed value $(s,t)$, where
$s,t$ are functions of $a,b,c$. We'll see this below.

This reweighting is analogous to performing a simple random walk
in $1$ dimension using a biased coin. The drift of the random
walk is a function of the bias of the coin. In our case we can think of
$a,b,c$ as a bias which affects the average slope (which 
corresponds to the drift). 
 
We computed $Z_n(a,b,c)$ above; one can find $(s,t)$
as a function of $(a,b,c)$ from this formula. It is easier to use the asymptotic expression for
$Z$ which we compute in the next section.

\begin{exer}
For $(s,t)$ on the boundary of the triangle of possible slopes,
describe the corresponding measure $\mu_{s,t}$. (Hint:
try the corners first.)
\end{exer}

\section{Uniform honeycomb dimers}

Recall the expressions
$$Z_{00}=\prod_{z^n=1}\prod_{w^n=1} a+bz+cw,$$
$$Z_{10}=\prod_{z^n=-1}\prod_{w^n=1} a+bz+cw,$$
$$Z_{01}=\prod_{z^n=1}\prod_{w^n=-1} a+bz+cw,$$
$$Z_{11}=\prod_{z^n=-1}\prod_{w^n=-1} a+bz+cw,$$
and (for $n$ odd)
$$Z=\frac12(Z_{00}+Z_{10}+Z_{01}-Z_{11}).$$

When $n$ is large, we can estimate these quantities
using integrals. Indeed, the logs of the right-hand sides
are different Riemann integrals for $\log(a+bz+cw)$.
Note that
$$\max_{\sigma,\tau}Z_{\sigma\tau}\leq Z\leq 2\max_{\sigma,\tau}Z_{\sigma\tau},$$
so the $n$th root of $Z$ and the $n$th root of the maximum of the $Z_{\sigma\tau}$
have the same limits.

It requires a bit of work \cite{CKP} to show that
\begin{theorem}
$$\lim_{n\to\infty}\frac1{n^2}\log Z = \frac{1}{(2\pi i)^2}
\int_{S^1}\int_{S^1}\mathrm{Log}(a+bz+cw)\frac{dz}{z}\frac{dw}{w}.$$
\end{theorem}
Here $\text{Log}$ denotes the principal branch. By symmetry
under complex conjugation the integral is real.
This normalized logarithm of the partition function is called
the {\bf free energy} $F$.

The difficulty in this theorem is that the integrand has two 
singularities (assuming $a,b,c$ satisfy the triangle
inequality) and the Riemann sums may become very small
if a point falls near one of these singularities\footnote{To see where
the singularities
occur, make a triangle with edge lengths $a,b,c$; think about it
sitting in the complex plane with edges $a,bz,cw$ where 
$|z|=|w|=1$. The two  possible orientations of triangle give
complex conjugate solutions $(z,w),(\bar z,\bar w).$}.  However
since the four integrals are Riemann sums on staggered
lattices, at most one of the four can become small.
(and it will be come small at both singularities by
symmetry under complex conjugation.)
Therefore either three or all four of the Riemann sums are actually good approximations to the integral. This is enough to
show that $\log Z$, when normalized, converges to the integral.

This type of integral is well-studied. It is known at the 
Mahler measure of the polynomial $P(z,w)=a+bz+cw$. 
We'll evaluate it for general $a,b,c$ below.
Note that when $a\ge b+c$ the integral can be evaluated 
quickly by residues and gives $F=\log a$.  Similarly
when $b>a+c$ or $c>a+b$ we get $F=\log b$ or $F=\log c$ respectively.

\subsection{Inverse Kasteleyn matrix}\label{IKM}

To compute probabilities of certain edges occurring in a dimer covering of a torus, one needs a linear combination of minors of
four inverse Kasteleyn matrices. 
Again there are some
difficulties due to the presence of zeros in the integrand, 
see \cite{KOS}.
In the limit $n\to\infty$, however, there is a simple expression
for the multiple edge probabilities. It is identical to the
statement in Corollary \ref{localstatcor}, except that we must
use the infinite matrix $K^{-1}$ defined by:
$$K^{-1}(\w_{0,0},\b_{x,y})=\frac1{(2\pi i)^2}\int \frac{z^{-y}w^{x}}{a+bz+cw}\frac{dz}{z}\frac{dw}{w}.$$
Here we are using special coordinates for the vertices;
$\w_{0,0}$ is the white vertex at the origin and
$\b_{x,y}$ corresponds to a black vertex at 
\newcommand{\e}{\text{e}}
$\e_1+x(\e_3-\e_1)+y(\e_1-\e_2)$, where $\e_1,\e_2,\e_3$
are the unit vectors in the directions of the three cube roots of $1$.

As an example, suppose that $a,b,c$ satisfy the triangle
inequality.
Let $\theta_a,\theta_b,\theta_c$ be the angles opposite
sides $a,b,c$ in a Euclidean triangle.
The probability of the $a$ edge $\w_{0,0}\b_{0,0}$
in a dimer covering of the 
honeycomb is
$$\Pr(\w_{0,0}\b_{0,0})=aK^{-1}(\w_{0,0},\b_{0,0})=
\frac{1}{4\pi^2}\int_{S^1}\int_{S^1} 
\frac{a}{a+bz+cw}\frac{dz}{iz}\frac{dw}{iw}.$$
Doing a contour integral over $w$ gives
$$\frac1{2\pi i}\int_{|a+bz|>c}\frac{a}{a+bz}\frac{dz}{z},$$
which simplifies to
$$=\left.\frac1{2\pi i}\log\left(\frac{bz}{a+bz}\right)\right|_{e^{-i(\pi-\theta_c)}}^{e^{i(\pi-\theta_c)}}$$
$$=\frac{\theta_a}{\pi}.$$

Note that this tends to $1$ or $0$ when the triangle degenerates,
that is, when one of $a,b,c$ exceeds the sum of the other two.
This indicates that the dimer covering becomes {\bf frozen}:
when $a\geq b+c$ only $a$-type edges are present.

\subsection{Decay of correlations}

One might suspect that dimers which are far apart are
uncorrelated: that is, the joint probability in a random dimer cover of two (or more)
edges which are far apart, is close to the product of their
individual probabilities. 
This is indeed the case (unless $s,t$ is on the boundary
of the triangle of allowed slopes), and the 
error, or correlation, defined by 
$\Pr(e_1\& e_2)-\Pr(e_1)\Pr(e_2)$, is an important
measure of how quickly information is lost with distance in the covering.
By Corollary \ref{localstatcor}, this correlation
is a constant times $K^{-1}(\b_1,\w_2)K^{-1}(\b_2,\w_1)$
if the two edges are $\b_1\w_1$ and $\b_2\w_2$.

The values of $K^{-1}(\w_{0,0},\b_{x,y})$ are Fourier
coefficients of $(a+bz+cw)^{-1}$. 
When $\b$ and $\w$ are far from each other, what can we say about
$K^{-1}(\w,\b)$?
The Fourier coefficients of an analytic function on the torus
$\{|z|=|w|=1\}$ decay exponentially
fast. However when $a,b,c$ satisfy the triangle inequality,
the function $(a+bz+cw)^{-1}$ is not analytic; it has two simple
poles (as we discussed earlier) on $\{|z|=|w|=1\}$. 

The size of the
Fourier coefficients is governed by the behavior 
at these poles, in fact the Fourier coefficients
decay linearly in $|x|+|y|$. 
As a consequence the correlation of distant edges 
decays {\em quadratically} in the distance between them.

This is an important observation which we use later as well:
the polynomial decay of correlations in the dimer model
is a consequence of the zeros of $a+bz+cw$ on the torus.
In more general situations we'll have a different polynomial
$P(z,w)$, and it will be important to find out where its zeros 
lie on the unit torus.

\subsection{Height fluctuations}

As in section \ref{IKM} above we can similarly compute, for $k\ne 0$,
$$K^{-1}(\w_{0,0},\b_{1,k})=-\frac{\sin(k\theta_b)}{\pi k b}.$$
By symmetry (or maybe there's an easy way to evaluate this directly, I don't know)
for $k\ne 0$,
$$K^{-1}(\w_{0,0},\b_{k,k})=-\frac{\sin(k\theta_a)}{\pi k a}.$$
Let $a_k=-\frac{\sin(k\theta_a)}{\pi k}$ for $k\ne 0$ and
$a_0=\theta_a/\pi$. 

Given any set of $k$ edges of type $a$ in the vertical
column passing through the edge $\w_{0,0}\b_{0,0}$,
say the edges $\w_{n_j,n_j}\b_{n_j,n_j}$ for $j=1,\dots k$,
by Corollary \ref{localstatcor}
the probability of these $k$ edges all occurring simultaneously is
$\det(a_{n_i-n_j})_{1\le i,j\le k}.$

So the presence of the edges in this column
forms a determinantal process with kernel $M_{i,j}=a_{i-j}$.
Determinantal processes \cite{Sosh} have the property that
the number of points in any region is a sum of independent 
(but not necessarily identically distributed) Bernoulli
random variables. The random variables have probabilities
which are the eigenvalues of the kernel $M$ restricted
to the domain in question. In particular if the variance in the number 
of points in an interval is large, 
this number is approximately Gaussian.

We can compute the variance in the number of points in an interval
of length $k$ as follows.
Let $M_k$ be the $k\times k$ matrix $M_{i,j}=a_{i-j}$ for $1\le i,j\le k$.
The sum of the variances of the Bernoullis is just the 
sum of $\lambda(1-\lambda)$ over the eigenvalues of $M_k$.
This is just $\text{Tr}(M_k(I_k-M_k))$ (where $I_k$ is the identity matrix).
With a bit of work (and the Fourier transform of the function
$f(\theta)=|\theta|(\pi-|\theta|)$ for $\theta\in[-\pi,\pi]$) one arrives at
\begin{equation}\label{logk}
\text{Tr}(M_k(I-M_k))=ka_0(1-a_0)-a_1^2(2k-2)-a_2^2(2k-4)-\dots
-a_{k-1}^2
\end{equation}
$$=\frac1{\pi^2}\log k + O(1).$$

We conclude that the variance in the height difference
between the face at the origin and the face at $(k,k)$,
that is, $k$ lattice spacings vertically away from the origin,
is proportional to the log of the distance
between the faces. In particular this allows us to
conclude that the height difference
between these points tends to a Gaussian when the points
are far apart. 

A similar argument gives the same height difference distribution
for any two faces at distance $k$ (up to lower order terms).

\begin{exer}
For the uniform measure $a=b=c=1$, compute the probability
that the face at the origin has three of its edges matched,
that is, looks like one of the two configurations in Figure 
\ref{hexflip}.
\end{exer}

\begin{exer}
Compute the constant term in the expression (\ref{logk}).
\end{exer}

\begin{exer}
How would you modify $M_k$ to compute the expected parity
of the height change from the face at the origin to the face
at $(k,k)$?
\end{exer}

\section{Legendre duality}

Above we computed the partition function for dimer coverings of the torus
with edge weights $a,b,c$. 
Here we would like to compute the number of dimer coverings of a torus
with uniform weights but with fixed slope $(s,t)$. 
Surprisingly, these computations are closely related.
Indeed, we saw that, with edge weights $a,b,c$, all coverings with
height change $h_x,h_y$ have the same weight $a^{n h_z}b^{n h_x}c^{n h_y}$ (we defined $h_z$ to be $h_z=n-h_x-h_y$).

So we just need to extract the coefficient $C_{h_x,h_y}$ in the expansion
of the partition function
$$Z_n(a,b,c)=\sum_{h_x,h_y}C_{h_x,h_y}a^{n h_z}b^{n h_x}c^{n h_y}.$$
This can be done as follows.
First, choose positive reals $a,b,c$ (if we can) so that the term
$C_{h_x,h_y}a^{n h_z}b^{n h_x}c^{n h_y}$ is the largest term. 
If we're lucky, this term and terms with nearby $h_x,h_y$ dominate
the sum, in the sense that all the other terms add up to a negligible amount
compared to these terms. In that case we can use the estimate
$$Z_n(a,b,c)\approx \sum_{|h_x/n-s|<\eps,|h_y/n-t|<\eps}C_{h_x,h_y}a^{n h_z}b^{n h_x}c^{n h_y}$$
and so $$\sum_{|h_x/n-s|<\eps,|h_y/n-t|<\eps}C_{h_x,h_y}\approx Z_n(a,b,c)a^{-n^2(1-s-t)}b^{-n^2s}c^{-n^2t}.$$

These estimates can in fact be made rigorous. One needs only check that
for fixed edge weights $a,b,c$, as $n\to\infty$ the height change concentrates on a fixed 
value $(s,t)$, that is, the {\bf variance} in the average slope $(s,t)$ tends to zero.

We can conclude that the growth rate (which we denote 
$-\sigma(s,t)$) of the number of stepped surfaces 
of fixed slope $(s,t)$ is 
\begin{eqnarray}\nonumber
-\sigma(s,t)&=&\lim_{n\to\infty}\frac1{n^2}\log C_{h_x,h_y}\\\nonumber
&=&\lim_{n\to\infty}\frac1{n^2}\left(\log Z_n - n^2(1-s-t)\log a-n^2 t\log b-n^2 s\log c\right)\\&=&\log Z-p_a\log a-p_b\log b-p_c\log c,
\label{Z-p}
\end{eqnarray}
where $(p_a,p_b,p_c)=(1-s-t,s,t)$ are the probabilities of the number of $a,b,c$-type edges, respectively. 

Here $-\sigma$ is the growth rate; $\sigma$ is called the {\bf surface
tension}, see below.

What we've done above is a standard operation, called 
{\bf Legendre duality}. 
Set $a=1$ and set $b=e^X,c=e^Y$.
Then our expression for the normalized logarithm of the partition
function is $$\log Z=\int_{|z|=|w|=1}\log(1+e^X z+e^Y w)\frac{dz}{2\pi iz}\frac{dw}{2\pi iw}.$$
We can rewrite this integral as
$$\log Z=R(X,Y)=\int_{|z|=e^X}\int_{|w|=e^Y}\log(1+z+w)\frac{dz}{2\pi iz}\frac{dw}{2\pi iw}.$$
Here $R(X,Y)$ is called the {\bf Ronkin function} of the polynomial $P(z,w)=1+z+w.$

We have shown in (\ref{Z-p}) 
above that the surface tension is the {\bf Legendre dual} of
$R$:
\begin{equation}\label{sigmaR}
-\sigma(s,t)=R(X,Y)-sX-tY,
\end{equation}
where $s=\frac{dR(X,Y)}{dX}$ and
$t=\frac{dR(X,Y)}{dY}.$
Both $R$ and $\sigma$ are convex functions; $R$ is defined for all $(X,Y)\in\R^2$
and $\sigma$ is only defined for $(s,t)$ in the triangle 
of allowed slopes $\{(s,t)~|~0\le s,t,1-s-t\le 1\}$.

Recall that $1-s-t,s,t$ are proportional to the angles of the triangle
with sides $a,b,c$. We have 
\begin{eqnarray}\nonumber
\frac{d\sigma}{ds}&=&X=\log b/a\\
\frac{d\sigma}{dt}&=&Y=\log c/a.\label{dsigma}
\end{eqnarray} 
{}From this we can easily find
\begin{theorem}\label{sigmaformula}
$$\sigma(s,t)=-\frac1{\pi}(L(\pi s)+L(\pi t)+L(\pi(1-s-t))),$$
where 
$$L(\theta)=-\int_0^\theta \log2\sin t\,dt$$
is the Lobachevsky function \cite{Milnor}.
\end{theorem}
 
Combined with (\ref{sigmaR}) this gives an expression for $R$ in terms of $L$ as well.
See Figures \ref{sigma} and \ref{Ronkin} for plots of $\sigma$
and $R$. 
\begin{figure}[htbp]
\center{\scalebox{1.0}{\includegraphics{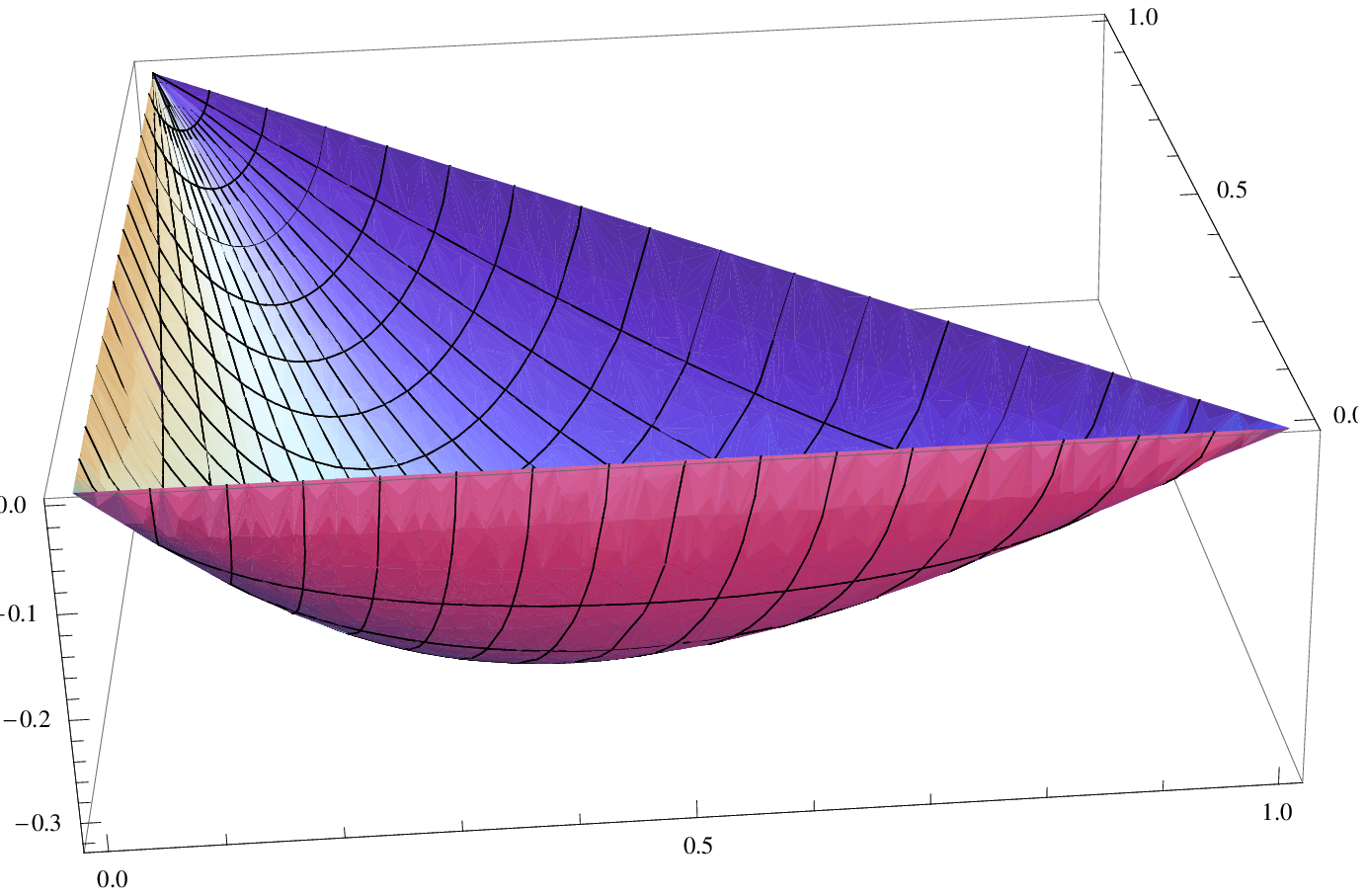}}}
\caption{\label{sigma}}
\end{figure}
\begin{figure}[htbp]
\center{\scalebox{1.0}{\includegraphics{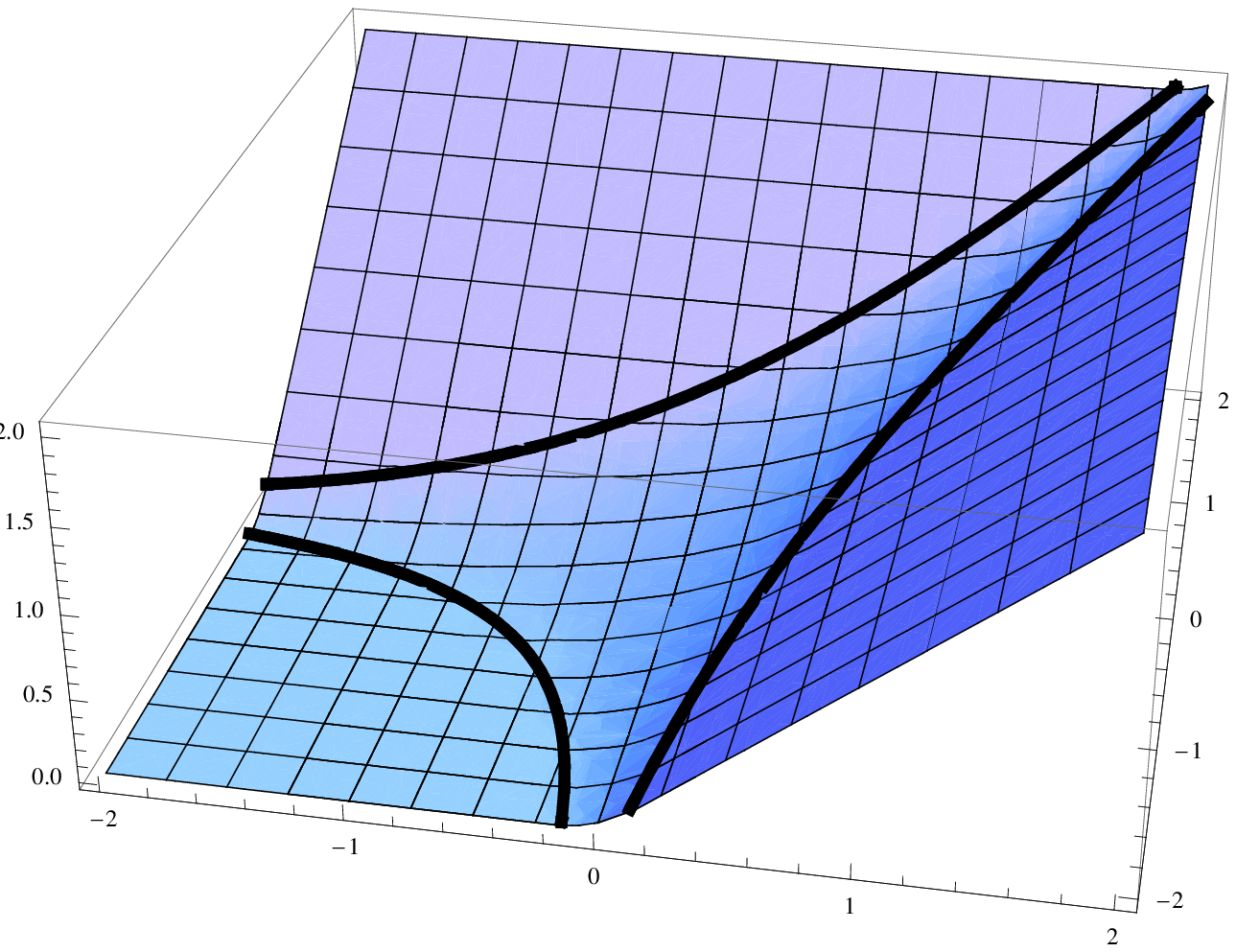}}}
\caption{\label{Ronkin}}
\end{figure}

\section{Boundary conditions}

As we can see in Figures \ref{lozengehexagon} 
and \ref{lozengeonside1}, boundary conditions have a large
influence on the shape of a random dimer configuration.

This influence is summed up in the following theorem.

\begin{theorem}[\bf{Limit shape theorem \cite{CKP}}]\label{CKP}
For each $\eps>0$ let $\gamma_\eps$ be a 
closed curve in $\eps\Z^3$
and which can be spanned by a stepped surface in
$\eps\Z^3$. Suppose the $\gamma_\eps$ converge 
as $\eps\to0$ to a closed curve $\gamma$. 
Then there is a surface $\Sigma_0$ 
spanning $\gamma$ with the following property.
For any $\delta>0$, with probability tending to $1$ as $\eps\to0$
a uniform  random stepped surface spanning $\gamma_\eps$
lies within $\delta$ of $\Sigma_0$. The surface $\Sigma_0$
is the graph of the unique function $h_0:P_{111}\to\R$
which minimizes
$$\min_h\int_U\sigma(\nabla h)dx\,dy,$$
where $$\sigma(s,t)=-\frac1{\pi}(L(\pi s)+L(\pi t)+L(\pi(1-s-t)))$$
and $U$ is the region enclosed by the projection of $\gamma$.
\end{theorem}

\begin{proof}
Here is a sketch of the proof of this theorem.
The space $L(\gamma)$ 
of Lipschitz functions with slope\footnote{that is, the slope at almost every point} in $\{(s,t)~|~0\leq s,t,1-s-t\le 1\}$
having fixed boundary values (i.e. so that their graph spans $\gamma$)
is compact in any natural topology, for example the $L^\infty$ metric. 
The integral of $\sigma(\nabla h)$ is a lower semicontinuous
functional on this space $L(\gamma)$, that is, given a 
convergent sequence of
functions, the surface tension integral of the limit is less than or equal
to the limit of the surface tension integrals. This follows from 
approximability of any Lipschitz function by 
piecewise linear Lipschitz functions on a fine triangular grid
(see the next paragraph).
The unicity statement follows from the convexity of $\sigma$.

Let $L_\eps(\gamma)$ be the set of stepped surfaces spanning $\gamma_\eps$.
Given a small $\delta>0$, take a finite cover of $L(\gamma)$ by
balls of radius $\delta$.
One can estimate the number of elements of $L_\eps(\gamma)$ 
contained in each ball as follows. This is essentially a large-deviation estimate.
Given a function $f\in L(\gamma)$, to estimate the number of
stepped surfaces of $L_\eps(\gamma)$  within $\delta$ of $f$, 
we triangulate $\Omega$ into equilateral
triangles of side $\sqrt{\eps}$. Because $f$ is Lipschitz,
Rademacher's theorem says that $f$ is differentiable almost everywhere.
In particular on all but a small fraction of these mesoscopic triangles, $f$ is nearly linear.
The number of surfaces lying near $f$ can then be broken up into
the number of surfaces lying near $f$ over each triangle, plus some errors
since these surfaces must glue together along the boundaries of the triangles.
It remains then to estimate the number of stepped surfaces lying close to 
a linear function on a triangle. 
This number can be shown, through a standard subadditivity argument,
to depend only on the slope of the triangle and its area, in the sense
that the number of stepped surfaces lying close to a linear function of slope $(s,t)$
on a triangle
of area $A$ is $\exp(-A\sigma(s,t)(1+o(1)))$ for some function $\sigma$. 

It remains then to compute $\sigma(s,t)$, which is minus the growth rate
of the stepped surfaces of slope $(s,t)$. This was accomplished for the torus
above, and again a standard argument shows that the value for the 
torus is the same as the value for the triangle or any other shape.
\end{proof}

\begin{exer}[up-right lattice paths] 
Consider the set of all up-right lattice paths in $\Z^2$ starting at $(0,0)$,
that is, paths in which each step is $(1,0)$ or $(0,1)$.
Given $a,b>0$ consider a measure on up-right paths 
of length $n$ which gives a path with $h$
horizontal steps and $v$ vertical steps a weight proportional to
$a^hb^v$. 
What is the partition function for paths of total length $n$?
What is the typical slope of a path of length $n$ for this measure?
What is the exponential growth rate of unweighted paths with average slope $s$?
Describe the Legendre duality relation here.
\end{exer}

\begin{exer}
On $\Z^2$, edges on every other vertical column have weight $2$
(other edges have weight $1$). 
Redo the previous exercise if the paths are weighted according to
the product of their edge weights times the $a^v b^h$ factor.
\end{exer}

\section{Burgers equation}

The surface tension minimization problem of Theorem \ref{CKP} above
can be solved as follows.
The Euler-Lagrange equation is 
\begin{equation}\label{EL}
\text{div}(\nabla\sigma(\nabla h))=0.
\end{equation}
That is, any surface tension minimizer will satisfy this equation locally,
at least where it is smooth. Here we should interpret this equation as follows.
First, $\nabla h=(s,t)$ is the slope, which is a function of $x,y$. Then
$\sigma(\nabla h)$ defines the local surface tension as a function of $x,y$. 
Now $\nabla\sigma$ is the gradient of $\sigma$ as a function of $s,t$.
By Legendre duality, see (\ref{dsigma}), we have
$\nabla\sigma(s,t)=(X,Y)$.
Finally the equation is that the divergence of this is zero,
that is $\frac{dX}{dx}+\frac{dY}{dy}=0.$
Substituting $X=\log b/a,Y=\log c/a$, see (\ref{dsigma}), this is
\begin{equation}\label{realEL}
\frac{d\log(b/a)}{dx}+\frac{d\log(c/a)}{dy}=0.
\end{equation}

Consider the triangle with sides $1,b/a,c/a$, placed in the complex plane so that the edge of length $1$ goes from $0$ to $1$
as shown in figure \ref{triangle}. 
\begin{figure}[htbp]
\scalebox{1.0}{\includegraphics{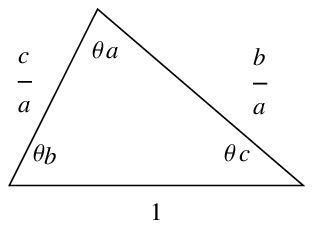}}
\caption{\label{triangle}}
\end{figure}
Define $z,w\in\C$ to be the other two edges 
of the triangle so that $1+z+w=0$ and $z=(b/a)e^{i(\pi-\theta_c)}$
and $w=(c/a)e^{i(\pi+\theta_b)}.$

Then $$\nabla h = (s,t)=(\frac{\theta_b}{\pi},\frac{\theta_c}{\pi})=
(\frac1{\pi}\arg(-w),\frac1\pi\arg(-\frac{1}z)).$$
In particular we have the consistency relation
$$s_y=h_{xy}=h_{yx}=t_x,$$
which gives 
\begin{equation}\label{imEL}
\im\left(\frac{z_x}{z}+\frac{w_y}{w}\right)=0.
\end{equation}

Combining (\ref{realEL}) and (\ref{imEL}) gives
\begin{theorem} For $z,w$ as defined above,
a general solution to the Euler-Lagrange equation (\ref{EL}) is given by
\begin{equation}\label{zxwy}
\frac{z_x}{z}+\frac{w_y}{w}=0.
\end{equation}
\end{theorem}

This equation can be solved using the method of
``complex characteristics''. The solutions
can be parametrized by analytic functions in two variables.

\begin{cor}\label{Qsoln} For any
solution to (\ref{zxwy}) there is a bivariate analytic function $Q_0$ for which
\begin{equation}\label{implicit}Q_0(z,xz+yw)=0.\end{equation} Conversely, any equation
of this type (with $Q_0$ analytic in both variables, which defines $z,w$ implicitly
as functions of $x,y$) gives a solution to (\ref{zxwy}).
\end{cor}

\begin{proof}
The existence of an analytic dependence between
$z$ and $xz+yw$ is equivalent to the 
equation 
$$\frac{z_x}{z_y}=\frac{(xz+yw)_x}{(xz+yw)_y},$$
or
$$z_x(xz_y+w+yw_y)=z_y(xz_x+z+yw_x).$$
However since $z$ and $w$ are analytically related, $z_xw_y=z_yw_x$,
leaving $$z_x w=z_y z.$$ Finally, since $z+w+1=0$, $z_y=-w_y$ and
so this last is equivalent to (\ref{zxwy}).
\end{proof}

\subsection{Volume constraint}

If we impose a volume constraint, that is, are interested in 
stepped surfaces with a fixed volume on one side, we can 
put a Lagrange multiplier in the minimization problem,
choosing to minimize instead $\int\int\sigma + \lambda\int\int h.$
This will have the effect of changing the Euler-Lagrange equation
to $$\text{div}(\nabla\sigma(\nabla h))=c$$ for a constant $c$.
Equation (\ref{zxwy}) then becomes 
\begin{equation}
\frac{z_x}{z}+\frac{w_y}{w}=c
\end{equation}
for the same constant $c$
and equation (\ref{implicit})
becomes 
\begin{equation}\label{implicitc}
Q(e^{-cx}z,e^{-cy}w)=0.
\end{equation}
For some reason the $c\ne0$ case has a more symmetric
equation than the $c=0$ case.

\subsection{Frozen boundary}

As we have discussed, and seen in the simulations, the minimizers
that we are looking for are not always analytic, in fact not 
even in general smooth. However Corollary \ref{Qsoln} seems to give 
$z,w$ analytically as functions of $x,y$.  What is going on is that
the equation (\ref{zxwy}) is only valid when our triangle
is well defined. When the triangle flattens out, that is, when $z,w$ become real,
typically one of the angles tends to $\pi$ and the other two to zero.
This implies that the slope $(s,t)$ is tending to one of the corners
of the triangle of allowed slopes, and we are entering a frozen phase.
The probabilities of one of the three edge types is tending to one,
and therefore we are on a facet.

The boundary between the analytic part of the limit surface and the facet
is the place where $z,w$ become real. This is called the {\bf frozen boundary},
and is described by the real locus of $Q$.

Note that when $z,w$ become real, 
the triangle can degenerate in one of three ways: the apex $-w$ can fall
to the left of $0$, between $0$ and $1$, or to the right of $1$.
These three possibilities correspond to the three different orientations of facets.

\subsection{General solution}

For general boundary conditions finding the analytic function
$Q$ in (\ref{Qsoln}) which describes the limit shape is difficult.

For boundary conditions resembling those in 
Figure \ref{lozengehexagon}, however, one can give an explicit answer.
Let $\Omega$ be a polygon with $3n$ edges in the
directions of the cube roots of $1$, in cyclic order 
$1,e^{2\pi i/3},e^{4\pi i/3},1,\dots$,
as in the regular hexagon or Figure \ref{heart}\footnote{The cyclic-order condition
can be relaxed by allowing some edges to have zero length}.
\begin{figure}[htbp]
\center{\scalebox{.5}{\includegraphics{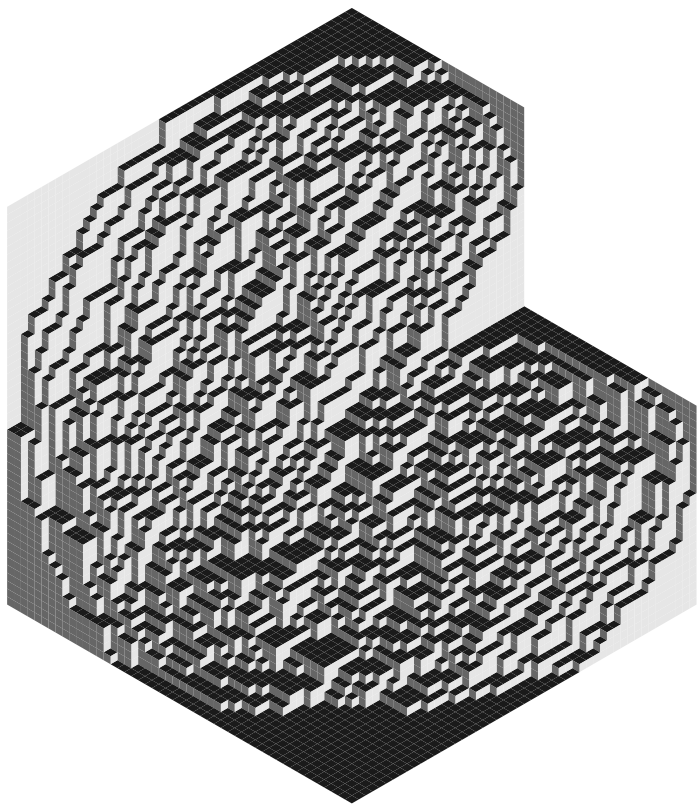}}}
\caption{\label{heart}In this case $Q$ has degree $3$ and
the frozen boundary is a cardiod.}
\end{figure}
Suppose that $\Omega$ can be spanned by (a limit of) stepped 
surfaces. Suppose that there is no ``taut edge'' in $\Omega$,
that is, every edge in the underlying graph has probability
lying strictly between $0$ and $1$ of occurring in a dimer cover
(an example of a region with a taut edge is the union
of two regular hexagons, joined along a side).
Then the limit shape arises from a {\bf rational plane curve} 
$$Q(z,w)=\sum_{0\le i,j,i+j\le n} c_{ij}z^iw^j$$
of degree $n$ (or $\leq n$ if there are edges of zero length). 
It can be determined 
by the condition that
its dual curve (the frozen boundary) 
is tangent to the $3n$ edges of $\Omega$, in order. 

Note that these polygonal boundary conditions can be used to approximate any boundary curve.

\begin{exer}
Consider stepped surfaces bounding the regular hexagon.
Show that plugging in $Q_0(u,v)=1+u+u^2-v^2$ into (\ref{implicit})
gives the solution to the limit shape for this boundary.

Let $Q(u,v)=1+u^2+v^2+r(u+v+uv),$ where $r>2$ is a parameter.
Show that this value of $Q$  in (\ref{implicitc})
gives a solution to the volume-constrained limit shape,
with volume which is a function of $r$.
\end{exer}

\section{Amoebas and Harnack curves}

As one might expect, everything we have done can be generalized
to other planar, periodic, bipartite graphs. Representative
examples are the square grid and the square-octagon grid (Figure \ref{squareoctagon}).

For simplicity, we're going to deal mainly with weighted
honeycomb dimers. Since it is possible, after simple
modifications, to embed any other periodic bipartite planar
graph in a honeycomb graph (possibly increasing the size of the
period), we're actually not losing any generality. We'll also illustrate 
our calculations in an example in section \ref{examplesection}.

So let's start with the honeycomb with a periodic
weight function $\nu$ on the edges, periodic with
period $\ell$ in directions $\hat x$ and $\hat y$. As in 
section \ref{Torus} we are led to an expression for the partition
function for the $n\ell\times n\ell$ torus:
 $$Z(H_{n\ell})=\frac12(Z_{00}+Z_{10}+Z_{01}-Z_{11})$$
where 
$$Z_{\tau_1,\tau_2}=\prod_{z^n=(-1)^{\tau_1}}\prod_{w^n=(-1)^{\tau_2}}P(z,w),$$
and where $P(z,w)$ is a polynomial, the {\bf characteristic polynomial},
with coefficients depending
on the weight function $\nu$.
The polynomial $P(z,w)$ is the determinant of $K(z,w)$,
the Kasteleyn matrix for the $\ell\times\ell$ torus (with 
appropriate extra
weights $z$ and $w$ on edges crossing fundamental domains);
as such it is just the signed sum of matchings on the $\ell\times\ell$ torus
consisting of a single $\ell\times\ell$ fundamental domain (with
the appropriate weight $(-1)^{h_xh_y}z^{h_x}w^{h_y}$).

The algebraic curve $P(z,w)=0$ is called the
{\bf spectral curve} of the dimer model, since it
describes the spectrum of the $K$ operator on the whole 
weighted honeycomb graph.

Many of the physical properties of the dimer model are
encoded in the polynomial $P$.
\begin{theorem}
The free energy is 
$$-\log Z:=-\lim_{n\to\infty}\frac1{n^2}\log Z(H_{n\ell})=-\frac1{(2\pi i)^2}\int_{|z|=|w|=1}\log P(z,w)\frac{dz}{z}\frac{dw}{w}.$$
The space of allowed slopes for invariant measures is the Newton
polygon of $P$ (the convex hull in $\R^2$ of the set $\{(i,j):z^iw^i 
\text{ is a monomial of $P$}\}$). The surface tension $\sigma(s,t)$
is the 
Legendre dual of the Ronkin function
$$R(X,Y)=\frac1{(2\pi i)^2}\int_{|z|=e^X}\int_{|w|=e^Y}\log P(z,w)\frac{dz}{z}\frac{dw}{w}.$$
\end{theorem}

We'll see more below.

\subsection{The amoeba of $P$}

The {\bf amoeba} of an algebraic curve $P(z,w)=0$ is the set
$$\A(P)=\{(\log|z|,\log|w|)\in\R^2~:~P(z,w)=0\}.$$
In other words, it is a projection to $\R^2$ of the zero set of $P$ in
$\C^2$, sending $(z,w)$ to $(\log|z|,\log|w|)$.
Note that for each point $(X,Y)\in\R^2$, the amoeba contains
$(X,Y)$ if and only if the torus $\{(z,w)\in\C^2~:~|z|=e^X,|w|=e^Y\}$
intersects $P(z,w)=0$. 

The amoeba has ``tentacles'' which are regions where
$z\to0,\infty$, or $w\to0,\infty$. 
Each tentacle is asymptotic to a line $\alpha\log |z|+\beta\log|w|+\gamma=0$. These tentacles divide the complement
of the amoeba into a certain number of unbounded complementary
components. There may be bounded complementary components
as well.

The following facts are standard; see \cite{PR,MikhR}.
The Ronkin function of $P$ is convex in $\R^2$, and
linear on each component of the complement of 
$\A(P)$\footnote{This shows that the complementary components
are convex.}.
The Legendre duality therefore maps each component
of the complement of $\A(P)$ to a single point of the Newton
polygon $\N(P)$. This is a point with integer
coordinates. Unbounded complementary components
correspond to integer points on the boundary of $\N(P)$;
bounded complementary components correspond
to integer points in the interior of $\N(P)$\footnote{Not every
integer point in $\N(P)$ may correspond to a
complementary component of $\A$.}.

See Figure \ref{amoebaexample} for an example of an amoeba of a spectral curve.

\subsection{Phases of EGMs}

Sheffield's theorem, Theorem \ref{Sheffthm}
says that to every point $(s,t)$ in the Newton
polygon of $P$ there is a unique ergodic Gibbs measure $\mu_{s,t}$.
The local statistics for a measure $\mu_{s,t}$ are determined by
the inverse Kasteleyn matrix $K_{X,Y}^{-1}$,
where $(X,Y)$ is related to $(s,t)$ via the Legendre duality,
$\nabla R(X,Y)=(s,t)$. As discussed in section \ref{IKM},
values of $K^{-1}$ are (linear combinations of) Fourier
coefficients of $1/P(z,w)$. In particular,
if $P(z,w)$ has no zeroes on the unit torus $\{|z|=|w|=1\}$,
then $1/P$ is analytic and so
its Fourier coefficients decay exponentially fast.
On the other hand if $P(z,w)$ has simple zeroes on the unit torus,
its Fourier coefficients decay linearly. 

This is exactly the condition which separates the different phases
of the dimer model. 
If a slope $(s,t)$ is chosen so that $(X,Y)$ is in 
(the closure of) an unbounded
component of the complement of the amoeba, 
then certain Fourier coefficients of $1/P$ (those contained in the appropriate dual cone) will vanish. This is enough to ensure that
$\mu_{s,t}$ is in a frozen phase (yes, this requires some argument
which we are not going to give here).
For slopes $(s,t)$ for which $(X,Y)$ is in (the closure of) a bounded
component of the complement of the amoeba, the 
edge-edge correlations decay exponentially fast (in all directions)
which is enough to show that the height fluctuations
have bounded variance, and we are in a gaseous (but not
frozen, since the correlations are nonzero) phase.

In the remaining case, $(X,Y)$ is in the interior of the amoeba,
and $P$ has zeroes on a torus. 
It is a beautiful and deep fact that the spectral curves 
arising in the dimer model are special in that $P$
has either two zeros, both simple, or a single node\footnote{A node 
is point where $P=0$ looks locally like the product of two lines, e.g. $P(x,y)=x^2-y^2+O(x,y)^3$ near $(0,0)$.} over each point in the
interior of $\A(P)$. As a consequence\footnote{We showed what happens in the case of a
simple pole already. The case of a node is fairly hard.} in this case the 
edge-edge correlations
decay quadratically (quadratically in generic directions---there may be directions
where the decay is faster). It is not hard to show that
this implies that the height variance
between distant points is unbounded, and we are in a 
liquid phase.

\subsection{Harnack curves}

Plane curves with the property described above,
that they have at most two zeros (both simple) or a single node
on each torus $|z|=constant,|w|=constant$
are called {\bf Harnack curves},
or simple Harnack curves. They were studied classically
by Harnack and more recently by Passare, Rullg\aa rd, Mikhalkin, 
and others \cite{PR, MikhR}.

The simplest definition is that a Harnack curve is a curve 
$P(z,w)=0$ with the property that the map from the zero set to
the amoeba $\A(P)$ 
is at most $2$ to $1$ over $\A(P)$. It will be
$2$ to $1$ with a finite number of possible exceptions 
(the integer points of $\N(P)$) on which the map may be $1$ to $1$.

\begin{theorem}[\cite{KOS,KOHarnack}]
The spectral curve of a dimer model is a Harnack curve.
Conversely, every Harnack curve arises as the spectral curve of 
some periodic bipartite weighted dimer model.
\end{theorem}

In \cite{KOHarnack} it was also 
shown, using dimer techniques, that the areas of complementary
components of $\A(P)$ and distances between tentacles
are global coordinates for the space of Harnack curves with
a given Newton polygon.

\subsection{Example}\label{examplesection}

Let's work out a detailed example illustrating the above
theory. We'll take dimers on the square grid with
$3\times 2$ fundamental domain (invariant under the lattice
generated by $(0,2)$ and $(3,1)$). 
Take fundamental domain with vertices labelled
as in Figure \ref{3X2}---we chose those weights to give us enough
parameters (5) to describe all possible gauge equivalence classes
of weights on the $3\times 2$ fundamental domain.
\begin{figure}[htbp]
\center{\scalebox{1}{\includegraphics{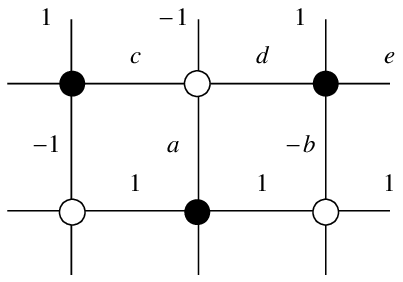}}}
\caption{\label{3X2}}
\end{figure} 
Letting $z$ be the eigenvalue
of translation in direction $(3,1)$ and $w$ be the eigenvalue
of translation by $(0,2)$, the Kasteleyn matrix (white white vertices corresponding
to rows and black to columns) is
$$K=\left(\begin{matrix}-1+\frac1w&1&\frac{e}{z}\\
c&a-w&d\\
\frac{z}{w}&1&-b+\frac1{w}\end{matrix}\right).$$
We have 
\begin{eqnarray*}
P(z,w)&=&\det K(z,w)\\
&=&
1+b+ab+bc+d+e-\frac{1+a+ab+c+d+ae}{w}+\frac{a}{w^2}-bw+\frac{ce}{z}+d\frac{z}{w}.
\end{eqnarray*}

This can of course be obtained by just counting
dimer covers of $\Z^2/\{(0,2),(3,1)\}$ with these weights,
and an appropriate factor $(-1)^{ij+j}z^iw^j$ when there are edges
going across fundamental domains.
Let's specialize to $b=2$ and all other edges of weight $1$.
Then $$P(z,w)=9-2w+\frac{1}{w^2}-\frac{7}{w}+\frac1{z}+\frac{z}{w}.$$
The amoeba is shown in Figure \ref{amoebaexample}.
\begin{figure}[htbp]
\center{\scalebox{1}{\includegraphics{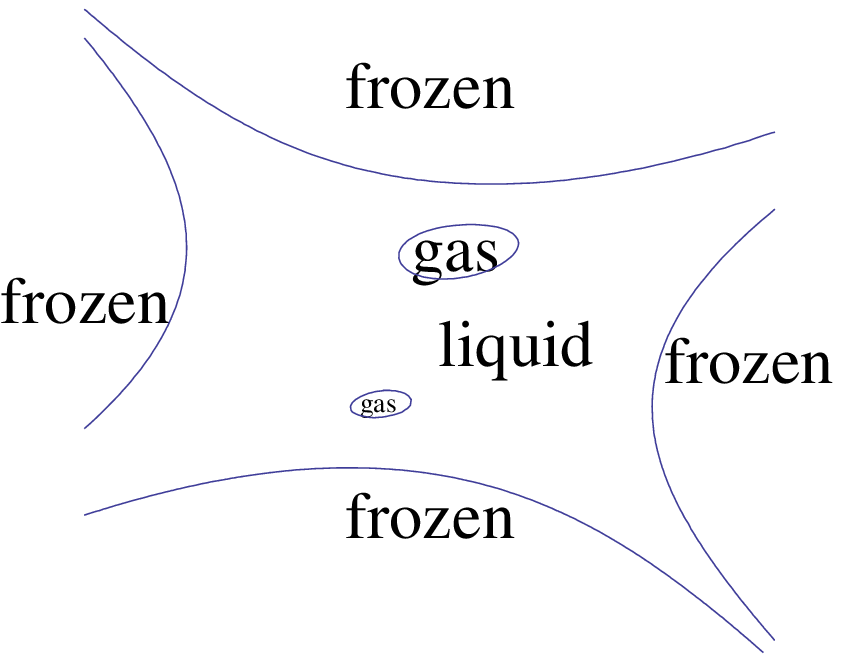}}}
\caption{\label{amoebaexample}}
\end{figure}
There are two gaseous components, corresponding
to EGMs with slopes $(0,0)$ and $(0,-1)$. The four frozen
EGMs correspond to  slopes $(1,-1),(0,1),(0,-2)$ and $(-1,0)$.
All other slopes are liquid phases.

For generic positive $a,b,c,d,e$ there are two gas components.
If we take $c=2,b=\frac12(3\pm\sqrt{3})$
and all other weights $1$ then there remains only one gaseous phase;
the other gas ``bubble'' in the amoeba 
shrinks to a point and becomes a node in $P=0$.
There is a codimension-$1$ 
subvariety of weights for which $P=0$ has a node.
Similarly there is a codimension-$2$ subvariety in which both gas 
bubbles disappear. For example
if we take all edge weights $1$ then both gaseous phases
disappear; we just have the uniform measure on dominos
again. Dimer models with no gas bubbles, that is,
in which $P$ has genus zero, have a number of other special 
properties, see \cite{K.isorad, KOHarnack}.

\begin{exer}
How would one model the uniform square-grid dimer model inside
the weighted honeycomb dimer model?
\end{exer}

\begin{exer}
On the square grid with uniform weights, 
take a fundamental domain consisting
of vertices $(0,0)$ and $(1,0)$. Show that (with an appropriate 
choice of coordinates) $P(z,w)=1+z+w-zw$. 
Sketch the amoeba of $P$. How many complementary
components are there? What frozen configurations do
they correspond to?
\end{exer}

\begin{exer}
Draw the amoebas for $P(z,w)=a+z+1/z+w+1/w$ for various 
values of $a\geq 4$ (when $|a|\ge 4$ this is a Harnack curve). 
Show that when $a=4,$ $P$ has a node
at $z=w=-1$. 
\end{exer}

\section{Fluctuations}

The study of the fluctuations of stepped surfaces away from their mean value
is not completed at present. 
Here we'll discuss the case of the whole plane, and stick to the
case of uniform honeycomb dimers with the maximal measure. 
Similar results
were obtained \cite{K.GFF} for uniform dimers on the
square lattice, and (harder) uniform weights on 
these lattices for other EGMs \cite{Kenyon.fluct}. This last paper
also computes the fluctuations for more general boundary
conditions, in particular when there is a nontrivial
limit shape. 

\subsection{The Gaussian free field}

The result is that the fluctuations are described by a Gaussian free field.
Here we discuss this somewhat mysterious object.

The Gaussian free field in two dimensions is a natural generalization of
one-dimensional Brownian motion. 
Recall that a Brownian bridge $\BB(t)$ is a Brownian motion on $\R$ started at the origin
and conditioned
to come back to its starting point after time $1$.
It is a random continuous function on $[0,1]$ which can be defined
by its Fourier series
$$\BB(t)=\sum_{k=1}^\infty \frac{c_k}{k}\sin(\pi k t),$$
where the coefficients $c_k$ are i.i.d. standard normals. 
We can consider the Brownian bridge to be a Gaussian measure on the 
infinite-dimensional space $L^2([0,1])$, and in the basis
of Fourier series the coefficients are independent.

The Gaussian free field on a rectangle has a very similar formula.
We write
$$\GFF(x,y) = \sum_{j,k=1}^\infty \frac{c_{j,k}}{\sqrt{j^2+k^2}}\sin(\pi j x)\sin(\pi k y),$$
where $c_{j,k}$ are i.i.d. standard normals. 
This defines the $\GFF$ as a Gaussian measure on a space of functions on $[0,1]^2$.
One problem with (or maybe we should say, feature of) this description  
is that the above series converges almost nowhere. 
And indeed, the $\GFF$ is not a random function but a random distribution. 
Given any smooth test function $\psi$, we can 
define $\GFF(\psi)$ by
$$\GFF(\psi) = \sum_{j,k=1}^\infty \frac{c_{j,k}}{\sqrt{j^2+k^2}}\int_{[0,1]^2}\psi(x,y)\sin(\pi j x)\sin(\pi k y)dx\,dy,$$
and this sum converges almost surely as long as $\psi$ is smooth enough
(in fact $\psi$ continuous suffices). Asking where the $\GFF$ lives, exactly, is a technical question;
for us it suffices to know that it is well defined when integrated against smooth
test functions.

The Gaussian free field on an arbitrary simply-connected bounded planar domain $\Omega$
has a similar description: it is a Gaussian process on (distributions on) $\Omega$
with the property that, when expanded in the basis of orthonormal eigenfunctions
of the Laplacian, has coefficients which are independent normals with
mean zero and variance $1/|\lambda|$, where $\lambda$
is the corresponding eigenvalue.

An alternative and maybe simpler description is that it is the Gaussian process with 
covariance kernel given by the Dirichlet Green's function $g(x,y)$.
That is, the $\GFF$ on $\Omega$ is the (unique) Gaussian measure
which satisfies 
$$\E(\GFF(z_1)\GFF(z_2))=g(z_1,z_2).$$
{}From this description we can see that the $\GFF$ is conformally invariant: given a conformal mapping $\phi\colon\Omega\to\Omega'$,
the Green's function satisfies
$g(z_1,z_2)=g(\phi(z_1),\phi(z_2)).$
This is enough to show that
$$\int_{\Omega}\GFF(z)\psi(z)|dz|^2=\int_{\Omega'}\GFF(w)\psi(\phi^{-1}(w))|dw|^2,$$ with the equality holding in distribution.

\subsection{On the plane}

The $\GFF$ on the plane has a similar formulation,
but it can only be integrated against $L^1$ functions of integral zero.
We have 
\begin{eqnarray*}
\lefteqn{\E((\GFF(z_1)-\GFF(z_2))(\GFF(z_3)-\GFF(z_4)))=}\hspace{2in}\\
&=&g(z_1,z_3)-g(z_1,z_4)-g(z_2,z_3)+g(z_2,z_4)\\
&=&-\frac1{2\pi}\log\left|\frac{(z_1-z_3)(z_2-z_4)}{(z_1-z_4)(z_2-z_3)}\right|,
\end{eqnarray*}
where the Green's function $g(z_1,z_2)=-\frac1{2\pi}\log|z_1-z_2|$.

\subsection{Gaussians and moments}

Recall that for a mean-zero multidimensional 
(even infinite dimensional) Gaussian process $X$,
if $x_1,\dots,x_n$ are linear functions of $X$ then
$\E(x_1\cdots x_n)$ is zero if $n$ is odd and if $n$ is even then
\begin{equation}\label{GFFmoments}
\E(x_1\cdots x_n)=\sum_\sigma \E(x_{\sigma(1)},x_{\sigma(2)})
\dots \E(x_{\sigma(n-1)},x_{\sigma(n)}),
\end{equation}
where the sum is over all pairings of the indices. 
For example
$$\E(x_1\cdots x_4)=\E(x_1,x_2)\E(x_3,x_4)+\E(x_1,x_3)\E(x_2,x_4)+
\E(x_1,x_4)\E(x_2,x_3).$$

This shows that the moments of order two, $\E(x_ix_j)$, where $x_i,x_j$ run over
a basis for the vector space,
determine a Gaussian process uniquely.
Another fact we will use is that any probability measure whose moments 
converge to those of a Gaussian, converges itself to a Gaussian 
\cite{Bil}.

\subsection{Height fluctuations on the plane}

We show here that the height fluctuations for 
the measure $\mu=\mu_{\frac13,\frac13}$ on dimer covers of the
honeycomb converge to the Gaussian free field. 
This is accomplished by explicitly computing the moments
$\E((h(z_1)-h(z_2))\dots(h(z_{n-1})-h(z_n)))$ and showing
that they converge to the moments of the $\GFF$. 

In fact we will only do the simplest case of the first
nontrivial moment. The calculations for higher moments are similar but 
more bookkeeping work is needed.

Let's fix four points $z_1,z_2,z_3,z_4\in\C$ and for each $\eps>0$
let $v_1,v_2,v_3,v_4$ be faces of $\eps\H$, the honeycomb scaled
by $\eps$, nearby.

In this section for convenience we will use the ``symmetric'' height function,
where $\omega_0$ is $1/3$ on each edge.
To compute $h(v_1)-h(v_2)$, we take a path in the dual graph
from $v_1$ to $v_2$ and count the number of dimers
crossing it, with a sign depending on whether the dimer has
white vertex on the left or right of the path.
The height difference $h(v_1)-h(v_2)$ is this signed number of dimers,
minus the expected signed number of dimers. 
When $\eps$ is  small, $v_1,v_2$ are many lattice spacings apart 
and we can choose a path which is polygonal, with edges in the 
three lattice directions. 
By linearity of expectation it suffices to consider the case when
both paths from $v_1$ to $v_2$ and from $v_3$ to $v_4$
are (disjoint) straight lines in lattice directions.
So let us consider first the case when both lines are vertical. 

Let $a_1,\dots,a_n$ be the edges crossing the first line (the line from
$v_1$ to $v_2$), and
$b_1,\dots,b_m$ be the edges crossing the second line (the line from $v_3$
to $v_4$).

Then 
\begin{eqnarray*}
\E[(h(v_1)-h(v_2))(h(v_3)-h(v_4))]&=&\sum_{i=1}^n\sum_{j=1}^m
\E\left[(\Id_{a_i}-\frac13)(\Id_{b_j}-\frac13)\right]\\
&=&\sum_{i=1}^n\sum_{j=1}^m
\E(\Id_{a_i}\Id_{b_j})-\frac19.
\end{eqnarray*}
Here $\Id_{a_i}$ is the indicator function of the presence of edge $a_i$,
and $\E(\Id_{a_i})=\E(\Id_{b_j})=\frac13$. 
This moment is thus equal to 
$$=\sum_{i=1}^n\sum_{j=1}^m K^{-1}(\b_i,\w'_j)K^{-1}(\b_j',\w_i)$$
where $a_i=\w_i\b_i$ and $\b_j=\w_j'\b_j'$.

At this point we need to use our knowledge of $K^{-1}(\b,\w)$ for
points $\b,\w$ far from each other.
We have 
\begin{lemma}\label{fundlemma}
$$K^{-1}(\w_{0,0},\b_{x,y})=\eps\re\left(\frac{e^{2\pi i(x-y)/3}}{\pi(e^{\pi i/6} x+e^{5\pi i/6} y)}\right)+
O(\frac{\eps}{|x|+|y|})^2.$$
\end{lemma}
Here the $\eps$ comes form the scaling of the lattice by $\eps$.

This lemma is really the {\bf fundamental calculation} in the whole
theory, so it is worth understanding.
First of all, recall that the values of $K^{-1}(\b,\w)$ are just the Fourier
coefficients for the function $1/(z+w+1)$. As mentioned earlier, 
if this function were 
smooth on the unit torus $|z|=|w|=1$, its Fourier coefficients
would decay rapidly. However 
$1/(z+w+1)$ has two simple poles on the torus:
at $(z,w)=(e^{2\pi i/3},e^{4\pi i/3})$ and its complex conjugate $(e^{4\pi i/3},e^{2\pi i/3})$.
The Fourier coefficients still exist, since you can integrate
a simple pole in two dimensions, but they decay only
linearly. Moreover, for $|x|+|y|$ large the $(x,y)$-Fourier coefficient
only depends on the value of the function $1/P$ near its poles.
Indeed the coefficient of the linearly-decaying term only depends on the 
first derivatives
of $P$ at its zeros. This implies that the large-scale 
behavior of $K^{-1}$---and hence the edge-correlations in the dimer
model---only depend on these few parameters (the location of the 
zeros of $P$ and its derivatives there).

With this lemma in hand we can compute (when the lattice is scaled by $\eps$)
\begin{eqnarray*}
\E[a_ib_j]&=&-\frac{\eps^2}{4\pi^2(u_2-u_1)}-\frac{\eps^2}{4\pi^2(\bar u_2-\bar u_1)^2}-
\frac{\eps^2e^{4\pi i(x-y)/3}}{4\pi^2|u_2-u_1|^2}-\frac{\eps^2e^{-4\pi i(x-y)/3}}{4\pi^2|u_2-u_1|^2}+\\
&&
+O(\eps^3/|u_1-u_2|^3),
\end{eqnarray*}
where $u_1$ is a point near $a_i$ and $u_2$ a point near $b_j$.
Summing over $i,j$, the terms with oscillating numerators are small,
and this becomes
\begin{equation}\label{integral}
2\re\int_{z_1}^{z_2}\int_{z_3}^{z_4}-\frac{1}{4\pi^2(u_1-u_2)^2}du_1\,du_2  + O(\eps)
\end{equation}
$$=-\frac{1}{2\pi^2}\log\frac{(z_2-z_4)(z_1-z_3)}{(z_2-z_3)(z_1-z_4)}+O(\eps).$$
Remarkably, we get the same integral (\ref{integral}) when the paths are pointing
in the other lattice directions, even when they are pointing in different directions.

\section{Open problems}

We have discussed many aspects of the dimer model.
There are many more avenues of research possible. We list a few of our favorites 
here.

\begin{enumerate}
\item Height mod $k$. What can be said about the random variable $e^{i\alpha h}$
where $\alpha$ is a constant and $h$ is the height function?
This is an analog of the spin-spin correlations in the Ising model
and is a more delicate quantity to measure that the
height function itself. Standard Toeplitz techniques allow one to 
evaluate it in lattice directions, and it is conjectured to be rotationally invariant
(for the uniform square grid dimers, say).
See \cite{Pinson} for some partial results. Can one describe the scaling limit
of this field?

\item All-order expansion. How accurately can one compute the
partition function for dimers in a polygon, such as that in
Figure \ref{heart}? For a given polygon, the leading asymptotics
(growth rate) as $\eps\to0$ is given by somewhat complicated integral
(that we don't know how to evaluate explicitly, in fact). 
What about the asymptotic series in $\eps$ of this partition function?
For the random $3D$ Young diagram, this series is important
in string theory. (Note that 
the partition function for the
uniform honeycomb dimer in a regular
hexagon has an exact form, see Exercise \ref{bppexer}). 

\item Bead model and Young tableaux. See \cite{Boutillier}.
For the $a,b,c$-weighted honeycomb dimer,
consider the limit $b=c=1, a\to0$. Under an appropriate rescaling the
limit is a continuous model, the bead model. The beads lie on parallel
stands and between any two bead on one strand there is a bead
on each of the adjacent strands. This model is closely 
related to Young tableaux. Can one carry the variational principle
over to this setting, getting a limit shape theorem 
(and fluctuations) for random Young tableaux?

\item Double-dimer model. 
Take two independent dimer covers of the grid, and superpose them. Configurations
consist of loops and doubled edges. Conjecturally, in the scaling limit
these loops are fractal and described by an $\text{SLE}_4$ process.
In particular their Hausdorff dimension is conjectured to be $3/2$. 

\end{enumerate}

\end{document}